\documentclass[preprint,12pt]{elsarticle}

\usepackage[margin=2cm]{geometry}
\usepackage[fleqn]{amsmath}
\usepackage{adjustbox}
\usepackage{amsfonts}  
\usepackage{dsfont}                  
\usepackage{amssymb}                    
\usepackage{graphicx}  
\usepackage{booktabs}
\usepackage{float}                 
\usepackage{epstopdf}
\usepackage[small,bf,hang]{caption}     
\usepackage{subcaption}                      
\usepackage{flafter}                    
\usepackage{listings}                   
\usepackage{pdfpages}
\usepackage[nottoc]{tocbibind}			%
\usepackage{eurosym}                    
\usepackage{textcomp}                   
\usepackage{fancyhdr}                   
\usepackage{emptypage}
\usepackage{longtable}
\usepackage{array}
\usepackage{booktabs}
\usepackage{comment}
\usepackage{geometry}
\usepackage{pdflscape}
\usepackage{rotating}
\usepackage{siunitx}
\usepackage{xparse}
\usepackage{interval}
\usepackage{cancel}
\usepackage{tikz}
\usepackage{mathrsfs}
\usepackage{lscape}
\usepackage{tabularx}
\usepackage[T1]{fontenc}
\newcommand{\R}{{\mathbb{R}}}

\usepackage{longtable}
\usepackage[inline]{enumitem}
\setlist[enumerate,1]{label=\textit{\alph*)}}

\usepackage{natbib}
\usepackage{tabularx}
\usepackage{amsthm}
\usepackage{mathtools}
\usepackage{multirow}
\usepackage{hhline}
\usepackage{algorithm}
\usepackage{algorithmicx}
\usepackage{algpseudocode}
\usepackage{titlesec}
\usepackage{accents}
\usepackage[titletoc,toc,title]{appendix}
\usepackage{hyperref}
\usepackage{upgreek}
\usepackage{lineno}
\newtheorem{theorem}{Theorem}[section]
\newtheorem*{assumption}{Assumptions}
\newtheorem*{rem}{Remark}
\newtheorem{cor}{Corollary}[section]
\newtheorem{lemma}{Lemma}[section]

\theoremstyle{definition}

\theoremstyle{remark}

\journal{Journal of Mathematical Analysis and Applications}

\begin{document}
	
	\begin{frontmatter}
		
		
		
		\title{Optimal Control of Medical Drug in a Nonlocal Model of Solid Tumor Growth
		}
		
		
		\author[inst1]{Bouhamidi Abderrahman}
		
		\affiliation[inst1]{organization={Université du littoral côte d'opale},
			addressline={50 rue Ferdinand Buisson, CS 80699}, 
			city={Calais},
			postcode={62100}, 
			country={France}}
		
		\author[inst2]{El Harraki Imad}
		\author[inst1,inst2]{Melouani Yassine}
		
		\affiliation[inst2]{organization={Ecole nationale supérieure des mines de
				Rabat},
			addressline={Av. Hadj Ahmed Cherkaoui - B.P. : 753}, 
			city={Rabat},
			postcode={10000}, 
			country={Morocco}}
		
		\begin{abstract}
			This paper presents a mathematical framework for optimizing drug delivery in cancer treatment using a nonlocal model of solid tumor growth. We present a coupled system of partial differential equations that incorporate long-range cellular interactions through integral terms and drug-induced cell death. The model accounts for spatial heterogeneity in both tumor cell density and drug concentration while capturing the complex dynamics of drug resistance development. We first establish the well-posedness of the coupled system by proving the existence and uniqueness of a solution under appropriate regularity conditions. The optimal control problem is then formulated to minimize tumor size while accounting for drug toxicity constraints. Using variational methods, we derive the necessary optimality conditions and characterize the optimal control through an adjoint system. Theoretical results can help to design effective chemotherapy schedules that balance treatment efficacy with adverse effects.
		\end{abstract}

		
	\end{frontmatter}
	
	
	\section*{Introduction}
	Cancer remains one of the leading causes of death worldwide, which requires
continued research on tumor growth dynamics and treatment strategies.
	Mathematical modeling has become an essential tool for understanding complex biological processes involved and optimizing therapeutic interventions\cite{gompertz1825xxiv, von1957quantitative}.
	The evolution of mathematical approaches to tumor growth reflects the increasing
complexity of our understanding. Early models focused primarily on describing volume dynamics, with the pioneering work employing the Gompertz
	model \cite{gompertz1825xxiv} and von Bertalanffy model
	\cite{von1957quantitative} establishing fundamental frameworks for
	characterizing growth kinetics. Although these approaches provided valuable
information, the growing recognition of tumor complexity led to more sophisticated
frameworks. In particular, the development of spatial models that incorporate the heterogeneous nature of the tumor microenvironment marked a significant advance.
	The work of Byrne and Chaplain \cite{byrne1996growth} demonstrated the importance of considering spatial aspects, particularly angiogenesis, in tumor development.
	Building on these spatial frameworks, successive modelling efforts have
	included genetic and epigenetic factors to better reflect current biological
	understanding. Altrock et al. \cite{altrock2015mathematics} developed
sophisticated approaches that account for clonal evolution within tumors,
providing critical insights into heterogeneity and treatment resistance
mechanisms. This understanding was further refined by Leschiera et al.
	\cite{leschiera2022mathematical}, whose work examining the relationship between
	cancer cell heterogeneity and immune response highlighted the importance of
	considering multiple biological scales. Although these advances significantly improved tumor growth modeling, considering only local cellular interactions overlooks interactions across various scales and distances within the tumor environment. Such limitations may result in incomplete models that do not capture the full complexity of tumor behavior, which is crucial for understanding growth patterns and effectively optimizing treatment strategies \cite{deisboeck2009collective}.
	This motivated the development of nonlocal modeling approaches capable of
	representing long-range cellular interactions. The groundbreaking work of
	Painter et al. \cite{painter2015nonlocal} demonstrated that nonlocal terms can
	effectively describe pattern formation in cancer growth. Building on this
	foundation, Buttenschön et al. \cite{buttenschon2018space} established a framework for nonlocal advection models in biological systems. The contributions of Armstrong et al. \cite{armstrong2006continuum} and
	Gerisch and Chaplain \cite{gerisch2008mathematical} further demonstrated the
	power of integro-differential equations in capturing complex tumor growth
	dynamics.
	Parallel to these developments in tumor growth modeling, the challenge of
	optimizing treatment protocols has led to an increasing application of control
	theory, particularly for chemotherapy administration. Beginning with Swan's\cite{swan1990role} demonstration of the applicability of control theory to chemotherapy scheduling. Martin \cite{martin1992optimal} extended these methods to address the crucial challenge of drug resistance, while De Pillis et al.
	\cite{de2008optimal} showed their utility in immunotherapy contexts. More recently, Li and You \cite{li2025optimal} extended this approach by developing optimal control strategies for reaction-diffusion models that explicitly incorporate tumor-immune interactions. The applicability of these optimal control approaches extends beyond cancer treatment to other biological systems, as demonstrated by Wu et al. \cite{wu2025analysis} in their analysis of hierarchical age-structured population models in polluted environments.

In this paper, we establish a mathematical framework for analyzing nonlocal tumor growth models with chemotherapy, establishing the existence and uniqueness of a solution. Then, we develop an optimal control theory for this system, deriving necessary optimality conditions through an adjoint formulation that accounts for the nonlocal effects. Furthermore, we obtain a characterization of the optimal control strategy that provides a guidance for drug dosing protocols. 
	
The remainder of this paper is organized as follows. In Section~\ref{sec1}, we present the nonlocal coupled system and the optimal control framework. In Section~\ref{sec2}, the notation and assumptions needed for our analysis are outlined. Section~\ref{sec3} establishes the well-posedness of the coupled system and proves the existence of an optimal control. Section~\ref{sec4} is devoted to deriving necessary optimality conditions and introducing an adjoint system to characterize the optimal control. In addition to the theoretical analysis, we present numerical experiments in Section~\ref{sec:5} to validate our theoretical findings and show the efficacy of the optimal control strategy. 
	\section{The Nonlocal Mathematical Model}\label{sec1}
	\subsection{Biological motivation}
	Cancer treatment outcomes are significantly influenced by the complex
	interactions between tumor cells and their microenvironment. These interactions
	occur not only through direct cell-to-cell contact but also via long-range
	signaling mechanisms that extend beyond immediate neighboring cells. Tumor cells
	interact through multiple mechanisms, including direct adhesion through
	membrane-bound molecules such as cadherins, paracrine signaling via diffusible
	factors, and mechanical forces transmitted through the extracellular matrix
	(ECM). As demonstrated by Armstrong et al. \cite{armstrong2006continuum} and
	Gerisch and Chaplain \cite{gerisch2008mathematical}, these interactions may be
	attractive or repulsive, significantly influencing both local tumor structure
	and overall growth patterns.
	
	The response to treatment represents another critical aspect of tumor dynamics.
	Blood-borne chemotherapy agents must navigate several barriers to reach tumor
	cells \cite{trachette1999mathematical}. These include diffusion through tissue,
	binding to target cells, and variable drug sensitivity across the tumor
	population. The effectiveness of treatment depends critically on both drug
	distribution and cellular response mechanisms, which can vary significantly
	across the tumor mass. Spatial heterogeneity plays a crucial role in treatment
	outcomes. Tumors exhibit significant variations in cell density, drug
	concentration, and microenvironmental conditions throughout their structure.
	This heterogeneity, coupled with the complex network of cellular interactions,
	can lead to varying treatment responses across different regions of the tumor.
	
	Our mathematical framework captures these biological complexities through a
	nonlocal velocity field $\mathrm{V}[w_p]$. Following the approaches developed by
	Painter et al. \cite{painter2015nonlocal} and Buttenschön et al.
	\cite{buttenschon2018space}, we assume that cells can detect and respond to
	their environment within a finite range. The velocity field $\mathrm{V}[w_p]$
	integrates the effects of adhesive forces between cells, chemical gradients, and
	mechanical stresses over a spatial domain determined by a finite sensing radius.
	The model incorporates drug response through a coupled reaction-diffusion system
	based on the work of Trachette and Byrne \cite{trachette1999mathematical}. This
	system accounts for the spatial distribution of drug concentration,
	cell-dependent drug uptake, and the development of resistance. 
	\subsection{Model Formulation}
	In this subsection, we will propose a mathematical model that describes the
	spatiotemporal dynamics of tumor growth under the influence of chemotherapy,
	incorporating nonlocal effects of cell-cell interactions. Let us denote by
	$I$ the drug concentration in the tumor vasculature depending on the time
	variable $t$. The value $I(t)$ is a control variable, constrained to lie within
	a biologically and clinically relevant set of admissible controls. We usually, consider the set of admissible controls as:
	\begin{equation}
		\label{Uad}
		U_{ad} = \Bigl\{u \in L^{\infty}(0,T) : 0 \leq u(t) \leq M_{tol} \text{ a.e. } t \in
		[0,T]\Bigr\},
	\end{equation}
	where $M_{tol}$ represents the maximum tolerable drug concentration, and $T$ is any positive value. This constraint
	set reflects physiological limitations and safety considerations in drug
	administration.
	
	The model consists of a coupled system of partial differential equations:
	\begin{equation}
		\label{system0}
		\left\{
		\begin{array}{lll}
			I \in U_{ad},\quad &&(a)\\
			\vspace{0.15cm}
			\dfrac{\partial p}{\partial t}(t,x)+\operatorname{div}
			\Bigl(\mathrm{V}\left[w_p (t,x)\right] 
			p(t,x)\Bigr)=F\bigl(p(t,x)\bigr)-C\bigl(d(t,x), p(t,x)\bigr), &t \in
			[0,T],\;x\in \mathbb{R}^N,&(b) \\
			\vspace{0.15cm}
			\dfrac{\partial d}{\partial t}(t,x)-\operatorname{div} \Bigl( D(x) \nabla
			d(t,x)\Bigr)=\Gamma\Bigl( I(t),d(t,x)\Bigr) , &t \in
			[0,T],\;x\in \mathbb{R}^N, &(c)\\
			p(0,x) = p_0(x), \quad d(0,x) = d_0(x),& x \in \mathbb{R}^N.&(d)
		\end{array}\right.
	\end{equation}
	Here, $p(t,x)$ represents the tumor cell density at position $x$ and time $t$,
	while $d(t,x)$ denotes the drug concentration. The function $\mathrm{V}\left[w_p
	\right]$ is a nonlocal velocity field, where $w_p$ is a convolution functional
	defined as follows:
	\begin{equation}
		\label{nonlocalterm}
		w_p(t,x)=\int_{\Omega} K(x,y)p(t,y) d y,
	\end{equation}
	where $K$ is a convolution kernel capturing long-range cell interactions, and
	$\Omega\subseteq \R^{N}$ is the domain of nonlocal sensibility, where
	$N\geq 1$ the spatial dimension, which is practically $N=1,2$. The mathematical model incorporates key biological and pharmacological mechanisms through carefully selected terms that capture the essential dynamics of tumor growth and drug treatment. The cell population dynamics are governed by two primary processes: cell proliferation, represented by the term $F(p)$, which typically follows either logistic or Gompertz growth patterns to account for resource limitations and competition among cancer cells, and drug-induced cell death, modeled by $C(d,p)$, which characterizes the interaction between local drug concentration and tumor cell density. The transport and metabolism of the therapeutic agent are described through two main mechanisms: First, the spatially-dependent diffusion coefficient $D(x)$ characterizes the heterogeneous drug transport through tumor tissue, reflecting variations in tissue density, extracellular matrix composition, and interstitial pressure. Second, a nonlinear function $\Gamma(I,d)$ describes the drug dynamics, incorporating various physiological processes including blood-tissue exchange, metabolic degradation, and clearance mechanisms. 
	
From a biological and physical perspective, both the tumor cell density $p(t,x)$ and drug concentration $d(t,x)$ must remain non-negative throughout the temporal and spatial domains. This fundamental constraint is not automatically guaranteed by the mathematical structure of system \eqref{system0}. This non-negativity preservation is crucial not only for the model's physical interpretation but also for the mathematical analysis of the optimal control problem, particularly when establishing the existence of optimal solutions. In Section \ref{existenceofOC}, we establish the non-negativity of the solutions $p$ and $d$ for all time $t\in[0,T]$, starting with non-negative initial conditions. 
	\subsection{Optimal Control Framework}
	To optimize the drug delivery protocol while accounting for both treatment
	efficacy and toxicity constraints, we formulate an optimal control problem. The
	objective is to minimize a cost functional that balances tumor reduction with
	drug-related adverse effects. Specifically, we consider the cost functional:
	\begin{equation}\label{Jcost}
	J(I) = \int_0^T \left[  \int_{\R^N} \alpha \,p(t,x) dx + \beta I(t)^2
	\right] dt + \gamma \int_{\R^N} p(T,x) dx.
\end{equation}
	This functional combines three key clinical considerations over the treatment
	period $[0,T]$. The first term measures the cumulative tumor density through the
	$L^1$ norm, weighted by a parameter $\alpha > 0$. The second term, weighted by
	$\beta > 0$, penalizes high drug concentrations, reflecting the need to minimize
	systemic toxicity. The terminal term, weighted by $\gamma > 0$, accounts for the
	final tumor size, which is important for the treatment effectiveness at the
	conclusion of therapy \cite{swan1990role}. Following the approaches similar to
	those developed by Ledzewicz and Schättler \cite{schattler2015optimal}, the
	optimal control problem thus becomes the minimization of \eqref{Jcost} subject to the
	control constraint system \eqref{system0}. Namely:
		\begin{equation}\label{OC_problem}
		\left\{
		\begin{array}{cc}
			\vspace{0.2cm}
	\min_{I\in U_{ad}} J(I),\\
	\textit{subject to}\quad \eqref{system0}. 
	\end{array}\right.
	\end{equation}
	\section{Notations and Assumptions}\label{sec2}
	In this section, we will first present some notations and assumptions which will be used throughout this paper. Let define the following spaces:
	
	Let $L^{\infty}(\R^N)$ be the space of essentially bounded measurable functions on $\R^N$, equipped with the norm:
	\[
	\|f\|_{L^{\infty}(\R^N)} = \operatorname{ess\,sup}_{x \in \R^N} |f(x)|.
	\]
	Let $C_{b}(\R^N)$ be he space of continuous bounded functions on $\R^N$, with the uniform norm:
	\[
	\|f\|_{C_{b}(\R^N)} = \sup_{x \in \R^N} |f(x)| .
	\]
	Let $C^1_b(\R^N)$ be the space of continuously differentiable bounded functions on $\R^N$ with bounded derivatives, normed by:
	\[
	\|f\|_{C^1_b(\R^N)} = \|f\|_{C(\R^N)} + \|\nabla f\|_{C(\R^N)}.
	\]
	For a Banach space $(X,\|\,.\,\|_X)$, let $C(0,T; X)$ denote the space of continuous functions from $[0,T]$ to $X$, with the norm:
	\[
	\|f\|_{C(0,T; X)} = \sup_{t \in [0,T]} \|f(t)\|_X  .
	\]
	Let $L^1(0,T; X)$ be the space of Bochner integrable functions from $[0,T]$ to $X$, with the norm:
	\[
	\|f\|_{L^1(0,T; X)} = \int_0^T \|f(t)\|_X dt.
	\]
	Let $W^{1,1}(\R^N)$ be the Sobolev space defined as:
	\[
	W^{1,1}(\R^N) = \{u \in L^{1}(\R^N) : D^\alpha u \in L^{1}(\R^N) \quad\text{for all}\quad |\alpha| \leq 1\}.
	\] 
	Furthermore, we state the following assumptions:
		\begin{assumption}
		\label{ass:model}
		We assume the following:
		\begin{enumerate}[label=\textbf{(A\arabic*)}]
			\item $D\in C_{b}^1(\mathbb{R}^N)$, such that $\forall x \in  \mathbb{R}^N $, $D_{0} \leq D(x) \leq D_{1}, \;\text{with} \;
			D_{0},D_{1}>0,$
			\item $d_0\in H^2(\mathbb{R}^N)\cap L^\infty(\mathbb{R}^N)$,
			\item $\Gamma:\mathbb{R} \times \mathbb{R}\to \R$ is Lipschitz function w.r.t second variable,
			\item There exists $g\in L^2(\mathbb{R}^N)\cap L^\infty(\mathbb{R}^N)$ such
			that:\\ $\forall (t,x)\in [0,T]\times \R^{N}$, $|\Gamma(I(t)),d(t,x))| \leq g(x)(1+|I(t)|+|d(t,x)|)$,
			\item $p_0\in L^{\infty}(\mathbb{R}^N)\cap W^{1,1}(\mathbb{R}^N)$,
			\item $F:\R \to \R$ and $C:\mathbb{R} \times \mathbb{R}\to \R$ are Lipschitz functions (w.r.t second variable for the function $C$),
			\item There exists $h\in L^1(\mathbb{R}^N)\cap L^\infty(\mathbb{R}^N)$ such
			that:\\ $\forall (t,x)\in [0,T]\times \R^{N}$, $|F(p(t,x))-C(d(t,x),p(t,x))| \leq h(x)(1+|p(t,x)|)$,
			\item $K\in C^2_b(\R^N\times\R^N)\cap L^1(\R^N\times \R^N)$, \quad(given in \eqref{nonlocalterm}),
			\item $\mathrm{V}\in C_b^2(\mathbb{R};\mathbb{R}^N).$
			\item $F \in C_{b}^1(\mathbb{R})$, $C \in C_{b}^1(\mathbb{R} \times \mathbb{R})$ and $\Gamma \in C_{b}^1(\mathbb{R} \times \mathbb{R})$.
			\item $K(x,y) = 0 \text{ for all } x \in \partial\Omega \text{ and } y \in \Omega$,
			\item $\mathrm{V}(0) = 0$ and $\mathrm{V}^{\prime}(0) = 0.$
		\end{enumerate}
	\end{assumption}
	\begin{rem}
		In fact, the assumption \textbf{(A10)} will be used only in Section~\ref{sec4} in order to ensure the derivability of the right hand side of the system \eqref{system1}.   
	\end{rem}
	\section{Existence and uniquness}\label{sec3}
	In this section, we establish the well-posedness of the coupled system
	\eqref{system0}. Our strategy consists of solving sequentially the parabolic
	equation for $d$ and then, with $d$ fixed, applying the theory developed in
	\cite{keimer2017existence} to solve the nonlocal equation for $p$. After
	establishing the well-posedness of our nonlocal system, we prove the existence
	of an optimal control for our problem.
	\subsection{Well-posedness of a week solution}
	Let us consider the coupled system $(b)-(c)-(d)$ in \eqref{system0} as:
		\begin{equation}
		\label{system1}
		\left\{
		\begin{array}{lll}
			\vspace{0.15cm}
			\dfrac{\partial p}{\partial t}(t,x)+\operatorname{div}
			\Bigl(\mathrm{V}\left[w_p (t,x)\right] 
			p(t,x)\Bigr)=F\bigl(p(t,x)\bigr)-C\bigl(d(t,x), p(t,x)\bigr), &t \in
			[0,T],\;x\in \mathbb{R}^N,&(b) \\
			\vspace{0.15cm}
			\dfrac{\partial d}{\partial t}(t,x)-\operatorname{div} \Bigl( D(x) \nabla
			d(t,x)\Bigr)=\Gamma\Bigl( I(t),d(t,x)\Bigr) , &t \in
			[0,T],\;x\in \mathbb{R}^N, &(c)\\
			p(0,x) = p_0(x), \quad d(0,x) = d_0(x),& x \in \mathbb{R}^N.&(d)
		\end{array}\right.
	\end{equation}
		\begin{theorem}[Well-posedness of the coupled system]
		\label{thm:wellposed}
		 Let the assumptions \textbf{(A1)-(A9)} hold. Then for $I\in L^{\infty}(\R^{N})$ fixed, the  coupled system \eqref{system1} has a unique solution $(p, d)$ in 
		$C(0,T; L^1(\mathbb{R}^N)) \times C(0,T; H^1(\mathbb{R}^N))$. Furthermore,
		we have the following estimates:
		\begin{equation}\label{L1est}
			\|p(t,\cdot)\|_{L^1(\mathbb{R}^N)} \leq C_{1}(\|p_0\|_{L^1(\mathbb{R}^N)} +
		T),
		\end{equation}
		\begin{equation}\label{Linfest}
			\|p(t,\cdot)\|_{L^\infty(\mathbb{R}^N)} \leq C_{2}(\|p_0\|_{L^\infty(\mathbb{R}^N)} +
			T),
		\end{equation}
		\begin{equation}\label{dl1}
			\left\| d(t,\cdot)\right\| _{L^2(\mathbb{R}^N)}  \leq C_3 (1+\left\| d_0\right\|
			_{L^2(\mathbb{R}^N)}+ \left\|I\right\|_{L^2(0,T)}),
		\end{equation}
		and
			\begin{equation}\label{dlinf}
			\|d\|_{L^\infty(0,T;L^\infty(\R^N))} \leq C_{4}\left(1+ \|d_0\|_{L^\infty(\mathbb{R}^N)}+
			\|I\|_{L^\infty(0,T)}\right) .
		\end{equation}
		where $C_{1}$, $C_{2}$, $C_{3}$, and $C_{4}$ are positive constants.
	\end{theorem}
	\begin{proof}

		\begin{itemize}
		\item We start by proving the well-posedness of the parabolic equation (6)-(c).\\ Let $\mathcal{A}d=-\operatorname{div}(D(x)\nabla d)$ with domain $D(\mathcal{A}) = H^2(\mathbb{R}^N)$. Under Assumption \textbf{(A1)}, this operator is uniformly elliptic. By Theorem 2.5.1 in Zheng \cite{zheng2004nonlinear}, under Assumptions \textbf{(A1)-(A3)}, there exists a unique solution 
		$
		d \in C(0,T; H^1(\mathbb{R}^N))
		$
		which can be represented as follows:
		\begin{equation}\label{duhamel}
				d(t,x) = e^{-t\mathcal{A}}d_0 + \int_0^t e^{-(t-s)\mathcal{A}}\Gamma(I(s),d(s,\cdot)))ds.
		\end{equation}
	For the $L^2$ estimate, we first note that $\{e^{-t\mathcal{A}}\}_{t\geq 0}$ is a strongly continuous semigroup on $L^2(\mathbb{R}^N)$ with the property:
		\[
		\|e^{-t\mathcal{A}}\|_{L^2 \to L^2} \leq 1,
		\]
		due to the uniform ellipticity of $\mathcal{A}$. By taking the $L^2$ norm of the mild solution we get:
		\begin{align*}
			\|d(t,\cdot)\|_{L^2(\R^N)} &\leq \|e^{-t\mathcal{A}}d_0\|_{L^2(\R^N)} + \int_0^t \|e^{-(t-s)\mathcal{A}}\Gamma(I(s),d(s,\cdot))\|_{L^2(\R^N)} ds\\
			&\leq \|d_0\|_{L^2(\R^N)} + \int_0^t \|\Gamma(I(s),d(s,\cdot))\|_{L^2(\R^N)} ds
		\end{align*}
			By Assumption \textbf{(A4)}, we have:
		\begin{equation*}
			\|\Gamma(I(s),d(s,\cdot))\|_{L^2(\R^N)} 
			\leq M_1(1 + \|I\|_{L^\infty(0,T)} + \|d(s,\cdot)\|_{L^2(\R^N)})
		\end{equation*}
		By inserting this into our previous inequality we have:
		\begin{align*}
			\|d(t, \cdot)\|_{L^2(\R^N)} &\leq \|d_0\|_{L^2(\R^N)} + M_1\int_0^t (1 + \|I\|_{L^\infty(0,T)} + \|d(s,\cdot)\|_{L^2(\R^N)}) ds\\
			&\leq \|d_0\|_{L^2(\R^N)} + M_1t(1 + \|I\|_{L^\infty(0,T)}) + M_1\int_0^t \|d(s,\cdot)\|_{L^2(\R^N)} ds
		\end{align*}
		Therefore, applying Gronwall's inequality leads to:
		\[
	\|d(t,\cdot)\|_{L^2(\R^N)} \leq (\|d_0\|_{L^2(\R^N)} + M_1T(1 + \|I\|_{L^\infty(0,T)}))e^{M_1T}
		\]
		This gives the estimate \eqref{dl1}. For the $L^\infty$ estimate, we use that the heat semigroup satisfies the maximum principle:
		\[
		\|e^{-t\mathcal{A}}\|_{L^\infty \to L^\infty} \leq 1,
		\]
		Then:
		\begin{align*}
\|d(t,\cdot)\|_{L^\infty(\R^N)} &\leq \|d_0\|_{L^\infty(\R^N)} + \int_0^t \|\Gamma(I(s),d(s,\cdot))\|_{L^\infty(\R^N)} ds\\
			&\leq \|d_0\|_{L^\infty(\R^N)} + M_2\int_0^t (1 + \|I\|_{L^\infty(0,T)} + \|d(s,\cdot)\|_{L^\infty(\R^N)}) ds
		\end{align*}
	 Therefore, applying Gronwall's inequality yields the estimate \eqref{dlinf}.
			\item Now, for the well-posedness of the nonlocal equation \eqref{system1}-(b), let $d \in C(0,T; H^1(\mathbb{R}^N))\cap 
		L^\infty((0,T);L^\infty(\mathbb{R}^N)) $ be fixed, then we will use Banach's fixed theorem and Keimer et al.\cite{keimer2018existence} to establish the desired result. In fact,
		Let $M:= 2\|p_0\|_{L^1(\R^N)}$, we define the following closed ball subset:
		\begin{equation}
			E = \{p\in C([0,T];L^1(\mathbb{R}^N)): \|p(t,\cdot)\|_{L^1(\R^N)} \leq M\},
		\end{equation}
		of the Banach space $C([0,T];L^1(\mathbb{R}^N))$.
		Furthermore, we consider the map $\Phi: E \to C([0,T];L^1(\mathbb{R}^N))$
		defined by:
\begin{equation}\label{eq:Phi}
			\Phi(p)(t,x) = p_0(\mathrm{X}[t,x](0))\det(D_x\mathrm{X}[t,x](0)) + \int_0^t
			[F(p) - C(d,p)](s,\mathrm{X}[t,x](s))\det(D \mathrm{X}[t,x](s))d s,
		\end{equation}
		where $\mathrm{X}[t,x]$ solves the characteristic equation:
\begin{equation}\label{eq:char}
			\begin{cases}
				\dfrac{\partial \mathrm{X}[t,x](s)}{\partial s} = \mathrm{V}\left[w_p(s,
				\mathrm{X}[t,x](s))\right], & s\in [0,T],\\
				\mathrm{X}[t,x](t) = x,
			\end{cases}
		\end{equation}
		and $\mathrm{V}$ is given in \eqref{system0}. By using the characteristics, for $p\in E$, and for
		any $t\in [0,T]$ we obtain:
		\begin{align*}
			\|\Phi(p)(t,\cdot)\|_{L^1(\R^N)} &\leq
\int_{\mathbb{R}^N}\Bigl|p_0(\mathrm{X}[t,x](0))\Bigr|\det(D_x\mathrm{X}[t,x](0))dx\\
			&+ \int_{\mathbb{R}^N}\int_0^t\Bigl|[F(p) - C(d,p)](\tau,\mathrm{X}[t,x](\tau))\Bigr|\det(D_x\mathrm{X}[t,x](\tau))d\tau d x.
		\end{align*}
		Now, the change of variables $y = \mathrm{X}[t,x](0)$ for the first term and $y =
		\mathrm{X}[t,x](\tau)$ for the second leads to:
		\begin{equation*}
			\int_{\mathbb{R}^N}|p_0(\mathrm{X}[t,x](0))|\det(D_x\mathrm{X}[t,x](0))dx =
			\|p_0\|_{L^1(\R^N)}.
		\end{equation*}
		For the second term by using the assumption $\mathbf{(A7)}$, we obtain:
		\begin{align*}
			\|\Phi(p)(t,\cdot)\|_{L^1(\R^N)} \leq&\|p_0\|_{L^1(\R^N)} +
			T\int_{\mathbb{R}^N}h(x)(1+|p(t,x)|)d x\\
			\leq& \|p_0\|_{L^1(\R^N)} + C_1T(1 + M),
		\end{align*}
		where $C_1=\max\left\lbrace \|h\|_{L^1(\R^N)}, \|h\|_{L^\infty(\R^N)}
		\right\rbrace$.
		By choosing $T$ small enough so that $C_1T(1 + M) \leq
		\|p_0\|_{L^1(\R^N)}$, we get for $t\in [0,T]$ the following:
		\begin{equation*}
			\|\Phi(p)(t,\cdot)\|_{L^1(\R^N)} \leq 2\|p_0\|_{L^1(\R^N)} = M.
		\end{equation*}
		Then $\Phi(E)\subseteq E$. As $F$ and $C$ are Lipschitz functions with the Lipschitz constants $L_F$ and $L_C$ respectively. Then for $p_1,p_2 \in E$ and for $t\in [0,T]$,we have:
		{\footnotesize
		\begin{align*}
			\|\Phi(p_1)(t,\cdot) - \Phi(p_2)(t,\cdot)\|_{L^1(\R^N)} &=
			\int_{\mathbb{R}^N}\int_0^t|F(p_1) - F(p_2) - (C(d,p_1) -
			C(d,p_2))|\cdot|\det(D_x\mathrm{X}[t,x](\tau))|d\tau dx\\
			&\leq (L_F+L_C)\int_0^t\|p_1(\tau,\cdot) - p_2(\tau,\cdot)\|_{L^1(\R^N)}d\tau.
		\end{align*}}
		Therefore:
		\begin{equation*}
			\|\Phi(p_1) - \Phi(p_2)\|_{C([0,T];L^1(\R^N))} \leq (L_F+L_C)T\|p_1 -
			p_2\|_{C([0,T];L^1(\R^N))}.
		\end{equation*}
		Then, by choosing $T^* = \min\left\lbrace T,\dfrac{1}{2(L_F+L_C)}\right\rbrace
		$, $\Phi$ is a contraction on $E$ for all $t$ in $ [0,T^*]$. By using Banach's fixed
		point theorem, there exists a unique $p^*\in E$ such that $\Phi(p^*)=p^*$. 
		
		Consequently, using assumptions \textbf{(A1)-(A9)} and according to Theorem 3.24 in Keimer et al. \cite{keimer2018existence}, we
		establish a unique solution for the nonlocal equation \eqref{system1}
		in $C([0,T^*];L^1(\mathbb{R}^N))$, furthermore, the existence and uniqueness of the solution holds for any
		$T\geq 0$.
		
		To obtain the estimates for $p$, we use the characteristics representation as
		follows, for all $t \geq t_{0}$ where $t_{0}\in [0,T]$:
		\begin{equation}\label{characterstics_rep}
			p(t,x) = p(t_{0},\mathrm{X}[t,x](t_{0}))\det(D_x\mathrm{X}[t,x](t_{0})) + \int_{t_{0}}^t [F(p) -
			C(d,p)](\tau,\mathrm{X}[t,x](\tau))\det(D_x\mathrm{X}[t,x](\tau))d\tau.
		\end{equation}
		By taking the $L^1$ norm and using the change of variables $y = \mathrm{X}[t,x](0)$ for the
		first term and $y = \mathrm{X}[t,x](\tau)$ for the second we get the following:
		\begin{align*}
			\|p(t,\cdot)\|_{L^1(\mathbb{R}^N)} &\leq \|p_0\|_{L^1(\mathbb{R}^N)} + \int_0^t\|F(p(\tau)) -
			C(d(\tau),p(\tau))\|_{L^1(\mathbb{R}^N)}d\tau.
		\end{align*}
		By using $\textbf{(A7)}$, we get:
		\begin{equation*}
			\|p(t,\cdot)\|_{L^1(\R^N)} \leq \|p_0\|_{L^1(\R^N)} + C_1\int_0^t(1 +
			\|p(\tau)\|_{L^1(\R^N)})d\tau.
		\end{equation*}
		Therefore, by using Gronwall's inequality, we get the desired $L^{1}$ estimate
		\eqref{L1est}. Similarly, to get the $L^\infty$ estimate we have:
		\begin{equation*}
			|p(t,x)| \leq |p_0(\xi[t,x](0))|\cdot|\det(D_x\mathrm{X}[t,x](0))|+ \int_0^t
			|F(p) -
			C(d,p)|(\tau,\mathrm{X}[t,x](\tau))\cdot|\det(D_x\mathrm{X}[t,x](\tau))|d\tau.
		\end{equation*}
		From Lemma 3.3 in \cite{keimer2018existence}, we have:
		\begin{equation}\label{det_positif}
			\det(D_x\mathrm{X}[t,x](\tau)) = \exp\left(\int_\tau^t-
			\operatorname{div}\Bigl( \mathrm{V}\bigl( w(s,\mathrm{X}[t,x](s))\bigr) \Bigr)
			ds\right).
		\end{equation}
		According to assumption $\textbf{(A9)}$, we have:
		\begin{equation*}
			|\det(D_x\mathrm{X}[t,x](\tau))| \leq \exp(LT),
		\end{equation*}
		where $L$ depends on bounds of derivatives of $\mathrm{V}$. Therefore:
		\begin{equation*}
			|p(t,x)| \leq \|p_0\|_{L^\infty(\mathbb{R}^N)}\exp(LT) + \int_0^t \exp(LT)|F(p)
			- C(d,p)|(\tau,\mathrm{X}[t,x](\tau))d\tau,
		\end{equation*}
		by taking the supremum over $x\in\mathbb{R}^N$, and using Gronwall's inequality,
		we get the desired $L^{\infty}$ estimate \eqref{Linfest}.
			\end{itemize}
	\end{proof}
	\subsection{Existence of an Optimal Control}\label{existenceofOC}
	Having established the well-posedness of our coupled system, we now turn our
	attention to the existence of an optimal control. We consider the following optimal control problem consisting on the minimization of the cost functional \eqref{Jcost} subject to the system \eqref{system0}.
	
	Before announcing our main result in this section, we need to establish the
	non-negativity of $p$ as given in the following lemma:
	\begin{lemma}[Non-negativity of the solutions]\label{lemma1}
		For $I \in \mathcal{U}_{ad}$ fixed, let the assumptions $\textbf{(A1)-(A9)}$ hold. If the initial conditions are non-negative $d_{0}\geq 0$ and $p_{0}\geq 0$, then for any $(t,x)\in [0,T]\times \mathbb{R}^N$, $p(t,x)$ and $d(t,x)$ are non-negative. 
	\end{lemma}
		\begin{proof}
			Let us first establish the non-negativity of $d$. Let $\mathcal{A}d = -\operatorname{div}(D(x)\nabla d)$. By the parabolic maximum principle, the semigroup ${e^{-t\mathcal{A}}}_{t\geq 0}$ is positivity-preserving: if $f \geq 0$, then $e^{-t\mathcal{A}}f \geq 0$ for all $t \geq 0$.
			We will proceed by contradiction. Let $t^*$ be the supremum of all times $t \in [0,T]$ such that $d(s,x) \geq 0$ for all $s \in [0,t]$ and $x \in \mathbb{R}^N$. Since $d_0 \geq 0$ and $d$ is continuous, $t^*$ is well-defined and positive. Suppose by contradiction that $t^* < T$. By continuity of $d$, we have $d(t^*,x) \geq 0$ for all $x \in \mathbb{R}^N$, and by maximality of $t^*$, there exist $\varepsilon > 0$ and $x^* \in \mathbb{R}^N$ such that $d(t,x^*) < 0$ for some $t \in [t^*, t^* + \varepsilon]$.
			For $t \in [t^*, t^* + \varepsilon]$, we use the variation of constants formula:
			\begin{equation}
				d(t,x) = e^{-(t-t^*)\mathcal{A}}d(t^*,\cdot) + \int_{t^*}^t e^{-(t-s)\mathcal{A}}\Gamma(I(s),d(s,\cdot))ds.
			\end{equation}
			Let $d^-(t,\cdot) = \max(-d(t,\cdot),0)$ denote the negative part of $d$. Since $\Gamma$ is Lipschitz in the second variable with constant $L$, we have:
			\begin{align*}
				\Gamma(I(s),d(s,\cdot)) &= \Gamma(I(s),0) + \int_0^{d(s,\cdot)} \phi(I(s),\tau)d\tau
			\end{align*}
			where $|\phi(I(s),\tau)| \leq L$ is the derivative of $\Gamma$ with respect to the second variable (which exists almost everywhere by Rademacher's theorem since $\Gamma$ is Lipschitz).
			By the growth condition we have \(|\Gamma(I(s),0)| \leq g(x)(1 + |I(s)|)\), therefore:
			\[\Gamma(I(s),d(s,x)) \geq -g(x)(1 + |I(s)|) - L|d(s,x)|.\]
			By using the positivity-preserving property of the semigroup and the non-negativity of $d(t^*,\cdot)$ we get:
			\begin{align*}
				d(t,x) &\geq -\int_{t^*}^t e^{-(t-s)\mathcal{A}}[g(\cdot)(1 + |I(s)|) + L|d(s,\cdot)|]ds\\
				&\geq -C_1\int_{t^*}^t (1 + \|d(s,\cdot)\|_{L^\infty(\R^N)})ds,
			\end{align*}
			where $C$ depends on $\|g\|_{L^\infty(\R^N)}$, $\|I\|_{L^\infty(0,T)}$, and $L$.
			Therefore:
			\begin{equation}\label{d-}
				\|d^-(t,\cdot)\|_{L^\infty(\R^N)}\leq C_1\int_{t^*}^t (1 + \|d(s,\cdot)\|_{L^\infty(\R^N)})ds.
			\end{equation}
			Now, in the region where $d$ is negative (i.e., for $t \in [t^*, t^* + \varepsilon]$), we have
			\(d(t,x) = -d^-(t,x)\),
			and therefore:
			\[\|d(s,\cdot)\|_{L^\infty(\R^N)} = \|d^-(s,\cdot)\|_{L^\infty(\R^N)}.\]
			By Substituting this back into \eqref{d-} and applying Gronwall's inequality we get the following:
			\begin{equation}
				\|d^-(t,\cdot)\|_{L^\infty(\R^N)} \leq C_1(t-t^*)\exp(C_1(t-t^*)).
			\end{equation}
				For sufficiently small $\varepsilon > 0$, this implies $d^-(t,\cdot) = 0$ for $t \in [t^*, t^* + \varepsilon]$, contradicting our assumption that $d$ becomes strictly negative after $t^*$. Therefore, $t^*  = T$ and $d$ remains non-negative on $[0,T]$.
			
			Now we will prove the non-negativity of $p$. For $t\in [0,T]$, let $p^+(t,\cdot)$ and $p^-(t,\cdot)$ denote the usual positive and negative parts of $p(t,\cdot)$, respectively. We have:
		$p^+(t,\cdot) := \max\left( p(t,\cdot),0\right)\geq 0$, and $p^-(t,\cdot) := \max\left( -p(t,\cdot),0\right) =- \min\left( p(t,\cdot),0\right) \geq 0$ and $p(t,\cdot)=p^{+}(t,\cdot)-p^{-}(t,\cdot)$.
		For any fixed $t^*\in [0,T]$ we
		have:
		\begin{equation*}
		\left\|p^-(t^*,\cdot) \right\|_{L^{1}(\R^{N})} =	\int_{\mathbb{R}^N} |p^-(t^*,x)|dx = -\int_{\{p(t^*,\cdot)<0\}} p(t^{*},x)dx.
		\end{equation*}
		By using the characteristic formula \eqref{characterstics_rep} and integrating on $\{p(t^*,\cdot)<0\}$, we get the
		following:
		\begin{align*}
			\left\|p^-(t^*,\cdot) \right\|_{L^{1}(\R^{N})} &= -\int_{\{p(t^*,\cdot)<0\}}
			p_0(\mathrm{X}[t^*,x](0))\det(D_x\mathrm{X}[t^*,x](0))dx\\
			&- \int_{\{p(t^*,\cdot)<0\}} \int_0^{t^*} [F(p) -
			C(d,p)](\tau,\mathrm{X}[t^*,x](\tau))\det(D_x\mathrm{X}[t^*,x](\tau))d\tau dx.
		\end{align*}
			Since $p_0\geq 0$ and according to \eqref{det_positif}, the determinant is non negative, then:
	$$
	\displaystyle	-\int_{\{p(t^*,\cdot)<0\}}
		p_0(\mathrm{X}[t^*,x](0))\det(D_x\mathrm{X}[t^*,x](0))dx \leq 0.
		$$
	By using the change of variables  $y = \mathrm{X}[t^*,x](\tau)$, the assumption (\textbf{A6}), and the estimate \eqref{L1est} respectively, we obtain: 
		 	$$
		 \left\|p^-(t^*,\cdot) \right\|_{L^{1}(\R^{N})}\leq \int_{\R^N} \int_0^{t^*}\Bigl| [F(p) -
		 	C(d,p)](\tau,\mathrm{X}[t^*,x](\tau))\Bigr|\det(D_x\mathrm{X}[t^*,x](\tau))d\tau dx.
		 	$$
	Using the assumption (\textbf{A6}) and the estimate \eqref{L1est}, we get:
		  $$\left\|p^-(t^*,\cdot) \right\|_{L^{1}(\R^{N})}\leq C_1t^{*}\left( 1 + C_{2}(\|p_0\|_{L^1(\mathbb{R}^N)} +
		 	T)\right) \leq  C^{\prime}t^{*},
		$$
		 where $C_1$, $C_1$, and $C^{\prime}$ are positive constants.
		  Let $\varepsilon >0 $, by choosing $t^*>0$ small enough such	that $C^{\prime}t^{*} < \varepsilon$, we get
		$
		\|p^-(t)\|_{L^1(\R^N)} \leq C^{\prime}t^{*} < \varepsilon
		$, for all $t\in [0,t^*]$.
		It follows that $p^-(t,\cdot)= 0$ a.e on $\R^{N}$, for all $t\in [0,t^*]$. Finally, $p(t,\cdot)\geq 0$ a.e on $\R^{N}$, for all $t\in [0,t^*]$. The continuity of $p$ ensures that $p(t,\cdot)\geq 0$  on $\R^{N}$, for all $t\in [0,t^*]$.

		Now, for all $t\in[t^*,2t^*]$, by using \eqref{characterstics_rep} we get similarly:
			\begin{equation}
			p(t,x) = p(t^{*},\mathrm{X}[t,x](t^*)) + \int_{t^{*}}^{t} [F(p) -
			C(d,p)](\tau,\mathrm{X}[t,x](\tau))\det(D_x\mathrm{X}[t,x](\tau))d\tau.
		\end{equation}
		Since $p(t^{^*},x)\geq 0$ for all $x\in \R^{N}$, by following the same arguments as in the interval $[0,t^{*}]$, we get:
		\begin{equation*}
			\|p^-(t,\cdot)\|_{L^1(\R^N)} \leq C^{\prime}(t-t^*)\leq C^{\prime}t^* < \varepsilon.
		\end{equation*}
		Therefore $p\geq 0$ on $[t^*,2t^*]$. Let $L$ be the greatest integer in $\mathbb{N}$ such that $Lt^*\leq T$. We
		can iterate the argument on each interval $[kt^*,(k+1)t^*]$ using the fact that; $p(kt^*)\geq 0$ as initial condition at each step for
		$k=0,1,\ldots,L-1$, and on $[Lt^{*},T]$ using $p(Lt^*)\geq 0$ as initial condition. Therefore we obtain the desired result $p(t,\cdot)\geq 0$ on $\R^{N}$ for all $t \in [0,T]$.
	\end{proof}
	
	Now, we can establish the result of existence of an optimal control:
	\begin{theorem}[Existence of an Optimal Control]
		\label{thm:optimal_control}
		Let the assumptions $\textbf{(A1)-(A9)}$ hold. Then there exists an optimal
		control $I^* \in \mathcal{U}_{ad}$ that minimizes the cost functional J.
	\end{theorem}
	\begin{proof}
		First, we establish that the cost functional $J$ is bounded below. Due to the
		non-negativity of $p$ (see Lemma~\ref{lemma1}), we have $\displaystyle J(I) \geq 0$ for all $I \in
		\mathcal{U}_{ad}$. This ensures that $\inf_{I \in \mathcal{U}_{ad}} J(I)$ is
		well-defined.
		Let $(I^n)_n \subset \mathcal{U}_{ad}$ be a minimizing sequence, i.e.,
		\begin{equation}
			\lim_{n \to \infty} J(I^n) = \inf_{I \in \mathcal{U}_{ad}} J(I).
		\end{equation}
		For each $n$, let $(p^n, d^n)$ be the  solution to system
		\eqref{system1} corresponding to the control $I^n$.
		From estimates \eqref{L1est},\eqref{Linfest} and \eqref{dl1} in Theorem~\ref{thm:wellposed}, we obtain the uniform bounds:
		\begin{equation}\label{Majorations_pn_dn}
			\left\| p^n\right\| _{	L^\infty(0,T;L^1(\mathbb{R}^N))}\leq C_{1},\;
			\left\| p^n\right\| _{	L^\infty(0,T;L^{\infty}(\mathbb{R}^N))} \leq C_{2},\;
			\left\| d^n\right\|
			_{L^2(0,T;L^2(\mathbb{R}^N))} \leq C_3. 
		\end{equation}
	where $C_{1}, C_{2}$, and $C_{3}$ are some positive constants, and we have also	$\left\| I^n\right\|
	_{L^\infty(0,T)} \leq M_{tol}$ where $M_{tol}$ is defined in \eqref{Uad}. By the Banach-Alaoglu theorem \cite{brezis2011functional}, there exist
		subsequences (still denoted by the same notations) and limit functions such that:
		\begin{align}
			I^n &\rightharpoonup^* I^* \quad \text{weakly}^* \text{ in } L^\infty(0,T),
			\label{weak*I}\\
			p^n &\rightharpoonup^* p^* \quad \text{weakly}^* \text{ in }
			L^\infty(0,T;L^\infty(\mathbb{R}^N)),\label{weak*p} \\
			d^n &\rightharpoonup d^* \quad \text{weakly in }
			L^2(0,T;L^2(\mathbb{R}^N)).\label{weak*d}
		\end{align}
		Now we will show that $(p^*,d^*)$ is the solution of the system \eqref{system1}-(b)-(c)-(d) for $I=I^*$. Indeed, to show that $p^{*}$ and $d^{*}$ satisfy the initial conditions (d)-\eqref{system1}, we consider a test function  $\phi \in C^1_c\left( [-T,T]\times\mathbb{R}^N\right) $, then integrating by part with respect to $t$:
		\begin{equation*}
			\int_0^T \int_{\mathbb{R}^N} \partial_t d^n(t,x) \phi(t,x)\; dxdt = \int_{\R^{N}}  d^n(0,x) \phi(0,x)dx	-\int_0^T \int_{\mathbb{R}^N}  d^n(t,x) \partial_t\phi(t,x)\; dxdt.
			\end{equation*}
	According to weak convergence \eqref{weak*d} and passing to the limit we get:
		\begin{equation*}
		\int_0^T \int_{\mathbb{R}^N} \partial_t d^*(t,x) \phi(t,x)\; dxdt = \lim_{n\to \infty}\int_{\R^{N}}  d^{*}(0,x) \phi(0,x)dx	-\int_0^T \int_{\mathbb{R}^N}  d^*(t,x) \partial_t\phi(t,x)\; dxdt.
	\end{equation*}
		Integrating again by part the second integral in the right hand side of last relation, we get:
			\begin{equation*}
		\lim_{n\to \infty}\int_{\R^{N}} \left[  d^{n}(0,x)-d^{*}(0,x) \right]\phi(0,x)dx=0.
		\end{equation*}
		As $d^{n}(0,x)=d_0(x)$, for all $n\geq 0$, then
		$$\int_{\R^{N}} \left[  d_{0}(x)-d^{*}(0,x) \right]\phi(0,x)dx=0,$$
		which shows that $d^{*}$ satisfies the initial condition. By using the same argument we have also $p^{*}$ satisfying the initial condition (d)-\eqref{system1}. Now it remains to show that $(p^*,d^*)$ is the solution of the system \eqref{system1}-(b)-(c).
		For all $\phi \in C^1_c\left( [0,T]\times\mathbb{R}^N\right) $, from equation (c)-\eqref{system1}, we have:
	\begin{equation*}
			\int_0^T \int_{\mathbb{R}^N} \partial_t d^n \phi\; dxdt + \int_0^T
			\int_{\mathbb{R}^N} D(x)\nabla d^n\cdot\nabla\phi\; dxdt = \int_0^T
			\int_{\mathbb{R}^N} \Gamma(I^{n},d^{n})\phi \; dxdt
		\end{equation*}
		Using the weak* convergence \eqref{weak*I}, the weak convergence \eqref{weak*d}, and the continuity of $\Gamma$ we obtain the limit equation:
	\begin{equation}\label{weak_conv_d}
	\displaystyle	\int_0^T \int_{\mathbb{R}^N} \partial_t d^* \phi \; dxdt + \int_0^T
			\int_{\mathbb{R}^N} D(x)\nabla d^*\cdot\nabla\phi \; dxdt = \int_0^T
			\int_{\mathbb{R}^N} \Gamma(I^{*},d^{*})\phi \; dxdt.
		\end{equation}
		For the nonlocal equation (b)-\eqref{system1}, for all $\phi \in C^1_c\left( [0,T]\times\mathbb{R}^N\right)$, we have: \begin{equation}\label{weak_conv_p1}
				\int_0^T \int_{\mathbb{R}^N}\partial_t p^n\phi \; dxdt \rightarrow \int_0^T
			\int_{\mathbb{R}^N} \partial_tp^*\phi \; dxdt.
		\end{equation}
		
	For the nonlocal term $w_{p^n}$, by using Fubini's theorem we have:
	\begin{align*}
		\int_0^T \int_{\R^{N}} w_{p^n}(t,x)\phi(t,x)dxdt =& \int_0^T \int_{\R^{N}} \int_{\Omega} K(x,y)p^n(t,y)\phi(t,x)dydxdt,\\
		=& \int_0^T \int_{\Omega} p^n(t,y)( \int_{\R^{N}} K(x,y)\phi(t,x)dx)dydt.
	\end{align*}
Let	define $\displaystyle \psi(t,y) =\int_{\R^{N}} K(x,y)\phi(t,x)dx$, by assumption \textbf{(A8)}, we have
	$$\int_0^T \int_{\R^N} |\psi(t,y)|dydt < \infty .$$
		Therefore $\psi \in L^1(0,T;L^{1}(\Omega))$, and by weak* convergence of $p^n$ in $L^\infty(0,T;L^\infty(\Omega))$ (which follows from the weak* convergence in $L^\infty(0,T;L^\infty(\mathbb{R}^N))$). It follows that:
$$	\int_0^T \int_{\Omega} p^n(t,y)\psi(t,y)dydt \rightarrow \int_0^T \int_{\Omega} p^*(t,y)\psi(t,y)dydt,$$
	which shows that $w_{p^n} \rightharpoonup^* w_{p^*}$ weakly* in $L^\infty(0,T;L^\infty(\mathbb{R}^N))$.	From assumption \textbf{(A9)}, the continuity of $\mathrm{V}$ ensures that $\mathrm{V}[w_{p^n}] \rightharpoonup^* \mathrm{V}[w_{p^*}]$ weakly* in $L^\infty(0,T;L^\infty(\mathbb{R}^N))$.
	Now, we have:
		\begin{align}
			&\int_0^T \int_{\mathbb{R}^N} p^n\mathrm{V}[w_{p^n}]\cdot\nabla\phi \; dxdt -
			\int_0^T \int_{\mathbb{R}^N} p^*\mathrm{V}[w_{p^*}]\cdot\nabla\phi \; dxdt \\
			&= \int_0^T \int_{\mathbb{R}^N} (p^n-p^*)\mathrm{V}[w_{p^n}]\cdot\nabla\phi \;
			dxdt + \int_0^T \int_{\mathbb{R}^N}
			p^*(\mathrm{V}[w_{p^n}]-V[w_{p^*}])\cdot\nabla\phi \; dxdt.
		\end{align}
		For the first term we have
		$\mathrm{V}[w_{p^n}]\cdot\nabla\phi \in L^1(0,T;L^1(\mathbb{R}^N))$ since
		$V$ is bounded, hence by the weakly convergence* \eqref{weak*p} we get:
		\begin{equation}
			\int_0^T \int_{\mathbb{R}^N} (p^n-p^*)\mathrm{V}[w_{p^n}]\cdot\nabla\phi \;
			dxdt \rightarrow 0.
		\end{equation}
		For the second term, since $V[w_{p^n}] \rightharpoonup^* V[w_{p^*}]$ weakly* in $L^\infty(0,T;L^\infty(\mathbb{R}^N))$ and
		$p^* \in L^\infty(0,T;L^\infty(\mathbb{R}^N))$ we have:
		\begin{equation}
			\int_0^T \int_{\mathbb{R}^N}
			p^*(\mathrm{V}[w_{p^n}]-\mathrm{V}[w_{p^*}])\cdot\nabla\phi \; dxdt \rightarrow
			0.
		\end{equation}
			Hence:
		\begin{equation}\label{weak_conv_p2}
			\int_0^T \int_{\mathbb{R}^N} p^n\mathrm{V}[w_{p^n}]\cdot\nabla\phi \; dxdt
			\rightarrow \int_0^T \int_{\mathbb{R}^N} p^*\mathrm{V}[w_{p^*}]\cdot\nabla\phi
			\; dxdt.
		\end{equation}
		For the right hand side, the use of the
		weak* convergence and the continuity of $F$ and $C$ (assumption $\mathbf{(A6)}$)
		allows passage to the limit:
\begin{equation}\label{weak_conv_p3}
			\int_0^T \int_{\mathbb{R}^N} (F(p^n) - C(d^n,p^n))\phi \; dxdt \rightarrow
			\int_0^T \int_{\mathbb{R}^N} (F(p^*) - C(d^*,p^*))\phi \; dxdt.
		\end{equation}
		From \eqref{weak_conv_d}, \eqref{weak_conv_p1}, \eqref{weak_conv_p2}, and \eqref{weak_conv_p3}, we conclude that $p^{*}$ and $d^{*}$ satisfy the system (b)-(c)\eqref{system1}.\\
		Now we will show that the cost functional $J$ is lower semicontinuous. For the space-time integral term, we use the non-negativity of $p$ (from \ref{lemma1}) and monotone convergence theorem to handle the integration over $\mathbb{R}^N$. This, combined with the weak* convergence of $p^n$ on bounded domains, leads to:
		\begin{equation*}
			 \int_0^T \int_{\mathbb{R}^N} \alpha p^*(t,x) dx dt\leq \liminf_{n \to \infty} \int_0^T \int_{\mathbb{R}^N} \alpha p^n(t,x) dx dt. 
		\end{equation*}
		The quadratic control term is convex, hence weakly lower semicontinuous, giving:
		\begin{equation*}
			\int_0^T \beta (I^*(t))^2 dt \leq \liminf_{n \to \infty} \int_0^T \beta (I^n(t))^2 dt.
		\end{equation*}
		The terminal term follows similarly to the first term. Therefore by combining these inequalities we get:
		\begin{equation*}
			J(I^*) \leq \liminf_{n \to \infty} J(I^n) = \inf_{I \in \mathcal{U}_{ad}} J(I),
		\end{equation*}
		which proves the optimality of $I^*$.
	\end{proof}
	\section{Optimality Conditions}\label{sec4}
	In this section, we will derive the necessary optimality conditions for our
	optimal control problem. Our approach follows the classical framework of optimal
	control theory for partial differential equations, as outlined in the seminal
	works of Lions \cite{lions1972some} and Tröltzsch \cite{troltzsch2010optimal}.
	The derivation of these conditions typically involves three main steps, starting
	by proving the differentiability of the control-to-state mapping, introducing an
	adjoint system, and deriving the optimality conditions.
	\subsection{Differentiability of the Control-to-State Mapping}
	Let us define the following control-to-state mapping:
	\begin{equation}
		\label{mappingG}
	\left\{
	\begin{array}{ll}
		G:\quad \mathcal{U}_{ad} &\longrightarrow C(0,T; L^1(\mathbb{R}^N)) \times C([0,T];
		H^1(\mathbb{R}^N)),\\ 
		\qquad \quad I &\longrightarrow G(I) = (p,d).
		\end{array}\right.
	\end{equation}
   where $(p,d)$ is the solution to system \eqref{system1} corresponding to a given $I\in U_{ad}$. Let us consider the notation $\dfrac{\partial G}{\partial h}(I)$ to denote the gâteaux derivative of $G$ at $I\in U_{ad}$ in the direction $h
   \in L^\infty(0,T)$.
	We have the following lemma:
	\begin{lemma}
		\label{lem:linearized_wellposed}
		Let Assumptions \textbf{(A1)-(A10)} hold. Then for any $h \in L^\infty(0,T)$, the linearized system:
		{\footnotesize
			\begin{equation}\label{linearized_system}
				\begin{cases}
					\vspace{0.15cm}
					\dfrac{\partial \xi}{\partial t}(t,x) + \operatorname{div}\Bigl(
					(\mathrm{V}\left[w_p \right]\xi)(t,x)\Bigr)  + \operatorname{div}\Bigl((
					\mathrm{V}^\prime\left[w_p \right]w_\xi p)(t,x)\Bigr)  =g(d,p)(t,x) \xi(t,x) - \partial_{1}C(d,p)(t,x)\eta(t,x), \quad &(a)\\
					\vspace{0.15cm}
					\dfrac{\partial \eta}{\partial t}(t,x) - \operatorname{div}(D(x)\nabla \eta(t,x))=\partial_{1}\Gamma(I,d)(t,x)h(t)+\partial_{2}\Gamma(I,d)(t,x)\eta(t,x), &(b)\\
					\xi(0,x) = 0, \quad \eta(0,x) = 0,
				\end{cases}
		\end{equation}}
		has a unique solution $(\xi, \eta)$ in $C(0,T; L^1(\mathbb{R}^N)) \times C(0,T; H^1(\mathbb{R}^N))$. Where $g(d,p):=F^{\prime}(p) - \partial_{2}C(d,p)$. $\mathrm{V}^{\prime}$, $F^{\prime}$ are the first derivatives of $\mathrm{V}$ and $F$ respectively, while $\partial_1$ and $\partial_2$ denote the first derivative with respect to first and second variable, respectively. 
	\end{lemma}
	\begin{proof}
		 The second equation \eqref{linearized_system}-(b) is a linear parabolic equation, then it has a unique solution. For the well-posedness of equation \eqref{linearized_system}-(a), the proof is based on the work of Keimer et al. \cite{keimer2018existence}. However, our equation \eqref{linearized_system}-(a) still slightly different from the one discussed in \cite{keimer2018existence}. Let us consider the following equation:
		{\footnotesize
			\begin{equation}\label{xi}
				\begin{cases}
					\dfrac{\partial \xi}{\partial t}(t,x) + \operatorname{div}\Bigl(
					(\mathrm{V}\left[w_p \right]\xi)(t,x)\Bigr)    = - \operatorname{div}\Bigl((
					\mathrm{V}^\prime\left[w_p \right]\mathrm{w} p)(t,x)\Bigr)- \partial_1C(d,p)(t,x)\eta(t,x)-g(d,p)(t,x) \xi(t,x),\\
					\xi(0,x) = 0,
				\end{cases}
		\end{equation}}
		where $\mathrm{w}$ is a fixed function in $C(0,T;C^{1}_{b}(\R^{N}))$ independent from $\xi$. The equation \eqref{xi} is a linear balance hyperbolic law. Hence, by using the Lemma~3.4 in \cite{keimer2018existence} and under the assumptions \textbf{(A1)-(A9)}, we deduce that \eqref{xi} has a unique solution which may be written as follows:
		\begin{equation}\label{xi_semi}
			\resizebox{\textwidth}{!}{$
				\displaystyle \xi(t,x)=-\int_0^t \left[ \partial_1C(d,p)\eta+\operatorname{div}\Bigl(
				\mathrm{V}^\prime\left[w_p \right]\,\mathrm{w}\, p\Bigr)\right] (s,\mathrm{X}[t,x](s))\exp\left(\int_{s}^{t}
				g(d,p)  (\tau,\mathrm{X}[t,x](\tau)) d \tau\right) 
				\det(D_{x} \mathrm{X}[t,x](s))d s,$}
		\end{equation}
		where $\mathrm{X}[t,x]$ is the solution of the characteristic equation \eqref{eq:char}.
		Now let us consider the following closed ball:
		\begin{equation}\label{setB}
			B=\left\lbrace f\in C\left( 0,T;C^{1}_{b}(\R^{N})\right):\; \|f\|_{C\left( 0,T;C^{1}_{b}(\R^{N})\right)}\leq  M \right\rbrace ,
		\end{equation} 
		where $M$ is a positive constant. Let consider the map $\mathrm{L}: B \to B$ defined by:
		\begin{equation}\label{Ftx} \resizebox{\textwidth}{!}{$
				\mathrm{L}(\mathrm{w})(t,x) =\displaystyle -\int_{0}^{t}\int_{\mathrm{X}[t,\Omega](\tau)} K(x,X[\tau,y](t)) \left[ \partial_1C(d,p)\eta+\operatorname{div}\Bigl(
				\mathrm{V}^\prime\left[w_p \right]\,\mathrm{w}\, p\Bigr)\right] (\tau,y)\exp\left(\int_{s}^{t}
				g(d,p)  (\tau,\mathrm{X}[\tau,y](s)) d s\right)d y d \tau,$}
		\end{equation}
		where $\Omega \subseteq \R^{N}$ is the domain appearing in \eqref{nonlocalterm}, and $\mathrm{X}[t,\Omega](\tau)$ is the image by $\mathrm{X}[t,\cdot](\tau)$ of $\Omega$.
		 We first show that the mapping $\mathrm{L}$ is well-defined, for $\mathrm{w} \in B$ we have:
		{\footnotesize\begin{align*}
				\left| \mathrm{L}(\mathrm{w})(t,x)\right|  &\leq \int_{0}^{t}\int_{\R^{N}}\left|  K(x,X[\tau,y](t))\left( \partial_1C(d,p)\eta\right)  (\tau,y)\exp\left(\int_{s}^{t}
				g(d,p)  (\tau,\mathrm{X}[\tau,y](s)) d s\right)\right| d y d \tau\\
				& + \int_{0}^{t}\int_{\R^{N}}\left|  K(x,X[\tau,y](t)) \operatorname{div}\Bigl(
				\mathrm{V}^\prime\left[w_p \right]\,\mathrm{w}\, p\Bigr) (\tau,y)\exp\left(\int_{s}^{t}
				g(d,p)   (\tau,\mathrm{X}[\tau,y](s)) d s\right)\right| d y d \tau.
		\end{align*}} 
		According to the higher regularity Theorem~4.2 in \cite{keimer2018existence} and the assumptions \textbf{(A5)}, \textbf{(A8)}, \textbf{(A9)} and \textbf{(A10)}, we obtain that $p \in C(0,T; W^{1,1}(\mathbb{R}^N))$. Then 
		{\footnotesize
		\begin{align*}
			 A_1:&=\left\| K\right\|_{L^{\infty}(\R^{N}\times \R^{N})} \left\|  \partial_1C(d,p)\right\|_{L^{\infty}(0,T,L^{\infty}(\R^{N}))} \left\| \eta\right\|_{L^{\infty}(0,T,L^{\infty}(\R^{N}))}   \exp\left( \left\| g(d,p) \right\|_{L^{1}(0,T,L^{\infty}(\R^{N}))}  \right)< \infty,\\
			 \overline{A_1}:&=\left\| \nabla_{1}K\right\|_{L^{\infty}(\R^{N}\times \R^{N})} \left\|  \partial_1C(d,p)\right\|_{L^{\infty}(0,T,L^{\infty}(\R^{N}))} \left\| \eta\right\|_{L^{\infty}(0,T,L^{\infty}(\R^{N}))}   \exp\left( \left\| g(d,p) \right\|_{L^{1}(0,T,L^{\infty}(\R^{N}))}  \right)< \infty,\\
			A_{2}:&= \left\| p \right\|_{C(0,T;L^{1}(\R^{N}))} \left\| \operatorname{div}(\mathrm{V}^{\prime}) \right\|_{C(0,T;L^{\infty}(\R^{N}))}+2\left\| \mathrm{V}^\prime\right\|_{C(0,T;L^{\infty}(\R^{N}))}\left\|\nabla p \right\|_{C(0,T;L^{1}(\R^{N}))}< \infty,
				\end{align*}}
are positive constants.
	We get after some usual estimates:
		{\footnotesize\begin{equation*}
				\Bigl|\mathrm{L}(\mathrm{w})(t,x)\Bigr|  \leq T\left( A_{1}+A_{2}\left\| K\right\|_{L^{\infty}(\R^{N}\times\R^{N})}  \exp\left( \left\| g(d,p) \right\|_{L^{1}(0,T,L^{\infty}(\R^{N}))}  \right)\left\| \mathrm{w}\right\|_{C\left( 0,T;C^{1}_{b}(\R^{N})\right)}\right),
		\end{equation*}} 
	and
		{\footnotesize\begin{equation*}
			\Bigl| \nabla\mathrm{L}(\mathrm{w})(t,x)\Bigr|  \leq T\left( \overline{A}_{1}+A_{2}\left\| \nabla K\right\|_{L^{\infty}(\R^N\times\R^{N}))}  \exp\left( \left\| g(d,p) \right\|_{L^{1}(0,T,L^{\infty}(\R^{N}))}  \right)\left\| \mathrm{w}\right\|_{C\left( 0,T;C^{1}_{b}(\R^{N})\right)}\right).
	\end{equation*}}
	Therefore, by taking the supremum over time and space, and choosing $T$ small enough we get:
		\begin{equation}
		\Bigl\|  \mathrm{L}(\mathrm{w})\Bigr\|_{C\left( 0,T;C^{1}_{b}(\R^{N})\right)} \leq M.
	\end{equation}
		
		Hence, we deduce that the mapping $\mathrm{L}$ is well-defined.
		Now let $\mathrm{w}_{1},\mathrm{w}_{2}\in B$, we get using similar estimations:
		\begin{equation}
			\Bigl\|  \mathrm{L}(\mathrm{w}_{1})- \mathrm{L}(\mathrm{w}_{2})\Bigr\| _{C\left( 0,T;C^{1}_{b}(\R^{N})\right)} \leq T R\left\| \mathrm{w}_{1}-\mathrm{w}_{2}\right\|_{C\left( 0,T;C^{1}_{b}(\R^{N})\right)}.
		\end{equation}
		where $R$ is a constant depending on $K,g,\mathrm{V}$ and $\eta$. Hence, for $T$ small enough, we deduce that $L$ is a contaction on $B$, which is a closed subset of the Banach space $C\left( 0,T;C^{1}_{b}(\R^{N})\right)$, then by Banach's fixed point theorem, there exists a unique function $\mathrm{w}^{*}\in B$ such that $\mathrm{w}^{*}=\mathrm{L}[\mathrm{w}^{*}]$.
		Substituting back into the equation \eqref{xi}, and by following the approach of \cite{keimer2018existence}, we get the existence of a solution for every $T>0$.\\
		For the uniqueness, let $(\xi_1,\eta_1)$ and $(\xi_2,\eta_2)$ be two solutions of the linearized system. Let define $\tilde{\xi} = \xi_1 - \xi_2$, and $\tilde{\eta} = \eta_1 - \eta_2$. Then for the parabolic equation $\tilde{\eta}$ we have:
		\begin{equation*}
			\begin{cases}
				\dfrac{\partial \tilde{\eta}}{\partial t} - \operatorname{div}(D(x)\nabla \tilde{\eta}) + \partial_{2}\Gamma(I,d)\tilde{\eta} = 0,\\
				\tilde{\eta}(0,x) = 0.
			\end{cases}
		\end{equation*}
By using estimates from \eqref{thm:wellposed}, we deduce that $\tilde{\eta}=0$ in $C(0,T,H^1(\R^N))$. In the other hand, for $\tilde{\xi}$ we have:
\begin{equation*}
	\begin{cases}
		\dfrac{\partial \tilde{\xi}}{\partial t} + \operatorname{div}\left(\mathrm{V}[w_p]\tilde{\xi}\right) + \operatorname{div}\left(\mathrm{V}'[w_p]w_{\tilde{\xi}} p\right) = F'(p)\tilde{\xi} - \partial_2C(d,p)\tilde{\xi},\\
		\tilde{\xi}(0,x) = 0.
	\end{cases}
\end{equation*}
By using the characteristics as in \eqref{xi_semi}, we obtain:
\begin{equation*}
		\tilde{\xi}(t,x) = -\int_0^t \operatorname{div}\left(\mathrm{V}'[w_p]w_{\tilde{\xi}} p\right)(s,X[t,x](s)) \exp\left(\int_{s}^{t} g(d,p)(\tau,X[t,x](\tau))d\tau\right)\det(D_x X[t,x](s))ds,
\end{equation*}
By taking the $L^1$ estimate as before we get:
\begin{equation*}
		\left\| \tilde{\xi}(t,\cdot)\right\| _{L^1(\R^N)}\leq \int_0^t \left\| \operatorname{div}\left(\mathrm{V}'[w_p]w_{\tilde{\xi}} p\right)\right\| _{L^1(\R^N)} \exp\left(\|g(d,p)\|_{L^\infty(\R^N)}(t-s)\right)ds.
\end{equation*}
Using assumptions \textbf{(A8)} and \textbf{(A9)}:
\begin{equation}
	 \left\| \operatorname{div}\left(\mathrm{V}'[w_p]w_{\tilde{\xi}} p\right)\right\| _{L^1(\R^N)} \leq C_1\|\tilde{\xi}(t,\cdot)\|_{L^1(\R^N)},
\end{equation}
where $C_1$ depends on $\|K\|_{C^1_b(\R^N)}$, $\|\mathrm{V}'\|_{C^1_b(\R^N)}$, and $\|p\|_{L^\infty(0,T,L^{\infty}(\R^N))}$.
Therefore:
\begin{equation*}
	\|\tilde{\xi}(t,\cdot)\|_{L^1(\R^N)} \leq C_1\int_0^t \|\tilde{\xi}(s,\cdot)\|_{L^1(\R^N)}\exp(C_2(t-s))ds.
\end{equation*}
By Gronwall's inequality we have 
	$\|\tilde{\xi}(t,\cdot)\|_{L^1(\R^N)} \leq 0,$
which leads to the desired result.
	\end{proof}
		\begin{lemma}[Gâteaux Differentiability of the Control-to-State Mapping]
			\label{lem:gateau_diff}
			Let Assumptions \textbf{(A1)-(A10)} hold. Then $G$ given by \eqref{mappingG} is Gâteaux differentiable at any $I \in \mathcal{U}_{ad}$. Furthermore, for any direction $h \in L^\infty(0,T)$, the Gâteaux derivative of $G$ at $I$ in the direction $h$ is $\dfrac{\partial G}{\partial h}(I) =(\xi, \eta)$, where $(\xi, \eta)$ is the solution to the linearized system \eqref{linearized_system}.
		\end{lemma}
			\begin{proof}
		Let $I \in \mathcal{U}_{ad}$, $h \in L^\infty(0,T)$, and $\varepsilon > 0$ such that $I^\varepsilon := I + \varepsilon h \in U_{ad}$.
		Let $(p^\varepsilon, d^\varepsilon)
		= G(I^\varepsilon)$ and $(p,d) = G(I)$, and let define $\xi^\varepsilon = (p^\varepsilon - p)/\varepsilon$ and
		$\eta^\varepsilon = (d^\varepsilon - d)/\varepsilon$.\\
		We aim to show that $(\xi^\varepsilon, \eta^\varepsilon) \to (\xi, \eta)$ in
		$X$ as $\varepsilon \to 0$. We have for difference quotients
		$(\xi^\varepsilon, \eta^\varepsilon)$:
		\begin{align*}
			\frac{\partial \xi^\varepsilon}{\partial t} +
			\operatorname{div}\left(\frac{\mathrm{V}\left[w_{p^\varepsilon}
				\right]p^\varepsilon - \mathrm{V}\left[w_p\right]p}{\varepsilon}\right) &=
			\frac{F(p^\varepsilon) - F(p)}{\varepsilon} -
			\frac{C(d^\varepsilon,p^\varepsilon) - C(d,p)}{\varepsilon}, \\
			\frac{\partial \eta^\varepsilon}{\partial t} - \operatorname{div}(D(x)\nabla
			\eta^\varepsilon)  &= \dfrac{\Gamma(I^{\varepsilon},d^{\varepsilon})-\Gamma(I,d)}{\varepsilon}, \\
			\xi^\varepsilon(x,0) = 0, \quad \eta^\varepsilon(x,0) &= 0.
		\end{align*}
		We have $w_{p^{\varepsilon}}=w_p+\varepsilon w_{\xi^{\varepsilon}}$. By assumptions \textbf{(A9)} and \textbf{(A10)}, Taylor's theorem with integral remainder for the nonlinear terms gives:
		\begin{align*}
			\frac{\mathrm{V}\left[w_{p^\varepsilon} \right]p^\varepsilon -
				\mathrm{V}\left[w_p\right]p}{\varepsilon} &=
			\mathrm{V}\left[w_p\right]\xi^\varepsilon +
			\mathrm{V}'\left[w_p\right]w_{\xi^\varepsilon}p + R_1^\varepsilon ,\\
			\frac{F(p^\varepsilon) - F(p)}{\varepsilon} &= F'(p)\xi^\varepsilon +
			R_2^\varepsilon, \\
			\frac{C(d^\varepsilon,p^\varepsilon) - C(d,p)}{\varepsilon} &=\partial_1 C(d,p)\eta^\varepsilon+ \partial_2C(d,p)\xi^\varepsilon+			R_3^\varepsilon,\\
			\dfrac{\Gamma(I^{\varepsilon},d^{\varepsilon})-\Gamma(I,d)}{\varepsilon} &=\partial_{1}\Gamma(I,d)h+\partial_{2}\Gamma(I,d)\eta^{\varepsilon}+R_4^\varepsilon,
		\end{align*}
		where $\left\|R_i^\varepsilon\right\|_{L^\infty(0,T;L^\infty(\mathbb{R}^N))}= o(\varepsilon)$, as $\varepsilon \rightarrow 0$, for $i=1,\dots,4$. Therefore we have:
			{\footnotesize
			\begin{equation}\label{linearized_system_epsilon}
				\begin{cases}
					\vspace{0.15cm}
						\dfrac{\partial \xi^\varepsilon}{\partial t} +
					\operatorname{div}\left(	\mathrm{V}\left[w_p\right]\xi^\varepsilon\right) +\operatorname{div}\left(
					\mathrm{V}'\left[w_p\right]w_{\xi^\varepsilon}p\right)  =F'(p)\xi^\varepsilon 
					 -\partial_1 C(d,p)\eta^\varepsilon -\partial_2C(d,p)\xi^\varepsilon		-\operatorname{div}\left(R_1^\varepsilon\right)+R_2^\varepsilon-R_3^\varepsilon, \quad &(a)\\
					\vspace{0.15cm}
						\dfrac{\partial \eta^\varepsilon}{\partial t} - \operatorname{div}(D(x)\nabla
					\eta^\varepsilon)=\partial_{1}\Gamma(I,d)h+\partial_{2}\Gamma(I,d)\eta^{\varepsilon}+R_4^\varepsilon, &(b)\\
					\xi^\varepsilon(0,x) = 0, \quad \eta^\varepsilon(0,x) = 0.
				\end{cases}
		\end{equation}}
			Now let $\tilde{\xi} = \xi^\varepsilon - \xi$ and $\tilde{\eta} =
		\eta^\varepsilon - \eta$. By subtracting \eqref{linearized_system} from \eqref{linearized_system_epsilon} we get the following perturbed system:
			\begin{equation}\label{linearized_system_tilde}
				\begin{cases}
					\vspace{0.15cm}
						\dfrac{\partial \tilde{\xi}}{\partial t} +
					\operatorname{div}(\mathrm{V}\left[w_p\right]\tilde{\xi}) +
					\operatorname{div}(\mathrm{V}^\prime\left[w_p \right]p w_{\tilde{\xi}})=
				g(d,p)\tilde{\xi} - \partial_1C(d,p)\tilde{\eta}  -\operatorname{div}\left(R_1^\varepsilon\right)+R_2^\varepsilon-R_3^\varepsilon, \quad &(a)\\
					\vspace{0.15cm}
						\dfrac{\partial \tilde{\eta}}{\partial t} - \operatorname{div}(D(x)\nabla
					\tilde{\eta})= \partial_{2}\Gamma(I,d)\tilde{\eta}+R_{4}^{\varepsilon}, &(b)\\
				\tilde{\xi}(x,0) = 0, \quad \tilde{\eta}(x,0) = 0.
				\end{cases}
		\end{equation}
			By using a similar technique as in the uniqueness proof of linearized system, we obtain:
		\[
		\left\|(\tilde{\xi}, \tilde{\eta})\right\|_X \leq
		C_1\left\|R^\varepsilon\right\|_{L^\infty(0,T;L^\infty(\mathbb{R}^N))},
		\]
		where $R^\varepsilon = R_2^\varepsilon+ \operatorname{div}(R_1^\varepsilon) + R_3^\varepsilon+R_4^\varepsilon$ and $C_1$ is some positive constant independent of $\varepsilon$. Therefore,
		\[
		\lim_{\varepsilon \to 0} \left\|\frac{G(I + \varepsilon h) - G(I)}{\varepsilon}
		- (\xi, \eta)\right\|_X = \lim_{\varepsilon \to 0} \left\|(\tilde{\xi},
		\tilde{\eta})\right\|_X = 0.
		\]
		This establishes the Gâteaux differentiability of $G$ at $I$ in the direction
		$h$, with $G'(I)h = (\xi, \eta)$.
			\end{proof}
	Now, we establish a result on the differentiability of the cost functional:
	\begin{theorem}[On the Gâteau-differentiability of the Cost Functional]
		\label{thm:cost_diff}
		Let Assumptions \textbf{(A1)-(A10)} hold. Then the cost functional $J$ is
		Gâteaux differentiable at any $I \in \mathcal{U}_{ad}$. Furthermore, for any
		direction $h \in L^\infty(0,T)$, the Gâteaux derivative of $J$ at $I$ in the
		direction $h$ is given by:
		\begin{equation}
			\label{dJ}
			\langle J'(I), h \rangle = \int_0^T \left[ \alpha \int_{\mathbb{R}^N}
			\xi(t,x)dx + 2 \beta I(t)h(t)\right]dt + \gamma \int_{\mathbb{R}^N}
			\xi(T,x)dx
		\end{equation}
		where $(\xi, \eta)$ is the solution to the linearized system
		\eqref{linearized_system}.
	\end{theorem}
{\begin{proof}
		Let $I \in \mathcal{U}_{ad}$ and define the functional $L_I: L^{\infty}(0,T) \to \mathbb{R}$ by:
		\begin{equation}
			\langle L_I,h\rangle = \int_0^T \left[ \alpha \int_{\mathbb{R}^N} \xi(t,x)dx + 2\beta I(t)h(t)\right]dt + \gamma \int_{\mathbb{R}^N} \xi(T,x)dx
		\end{equation}
		where $(\xi,\eta)$ solves the linearized system \eqref{linearized_system} with the direction $h$. Now let $h_1, h_2 \in L^{\infty}(0,T)$ and $\lambda \in \mathbb{R}$, and let us consider $h = h_1 + \lambda h_2$, with $(\xi_1,\eta_1)$ and $(\xi_2,\eta_2)$ are the solutions to the linearized system corresponding to directions $h_1$ and $h_2$ respectively. By the linearity of system \eqref{linearized_system}, the solution $(\xi,\eta)$ corresponding to $h$ satisfies $\xi = \xi_1 + \lambda \xi_2$ and $\eta = \eta_1 + \lambda \eta_2$. Consequently:
		\begin{align*}
			\langle L_I,h_1 + \lambda h_2\rangle &= \int_0^T \left[ \alpha \int_{\mathbb{R}^N} (\xi_1 + \lambda \xi_2)dx + 2\beta I(t)(h_1(t) + \lambda h_2(t))\right]dt  + \gamma \int_{\mathbb{R}^N} (\xi_1 + \lambda \xi_2)(T,x)dx \\
			&= 	\langle L_I,h_1\rangle + \lambda	\langle L_I,  h_2\rangle.
		\end{align*}
		Therefore $L_I$ is linear. Furthermore, for any $h \in L^\infty(0,T)$, by using the estimates from Lemma \ref{lem:linearized_wellposed} we obtain:
		\begin{equation*}
			|\langle L_I,h\rangle| \leq \alpha \int_0^T \int_{\mathbb{R}^N} |\xi(t,x)|dx dt + 2\beta \|I\|_{L^\infty}\|h\|_{L^\infty}T + \gamma \int_{\mathbb{R}^N} |\xi(T,x)|dx \leq C_1 \|h\|_{L^\infty(0,T)},
		\end{equation*}
		where $C_1$ is a positive constant depending on $\alpha$, $\beta$, $\gamma$, $T$, and the bounds from Lemma \ref{lem:linearized_wellposed}. Therefore $L_I$ is continuous.
		 
		To show that $L_I$ represents the Gâteaux derivative of $J$, let $h \in L^{\infty}(0,T)$, $\varepsilon > 0$, and let define $I^\varepsilon = I + \varepsilon h$ and let $(p^\varepsilon, d^\varepsilon) = G(I^\varepsilon)$ and $(p,d) = G(I)$. Let define $\Phi_I(\varepsilon)=	\dfrac{J(I^\varepsilon) - J(I)}{\varepsilon}$, we have:
		\begin{align*}
			\Phi_I(\varepsilon) &= \int_0^T \left[\alpha \int_{\mathbb{R}^N} \frac{p^\varepsilon - p}{\varepsilon}dx + \beta\frac{(I^\varepsilon)^2 - I^2}{\varepsilon}\right]dt  + \gamma \int_{\mathbb{R}^N} \frac{p^\varepsilon(T,x) - p(T,x)}{\varepsilon}dx \\
			&= \int_0^T \left[\alpha \int_{\mathbb{R}^N} \xi^\varepsilon dx + \beta(2Ih + \varepsilon h^2)\right]dt + \gamma \int_{\mathbb{R}^N} \xi^\varepsilon(T,x)dx
		\end{align*}
		where $\xi^\varepsilon = (p^\varepsilon - p)/\varepsilon$. By Lemma \ref{lem:linearized_wellposed}, we have $\xi^\varepsilon \to \xi$ in $X$ as $\varepsilon \to 0$, where $\xi$ is the solution to the linearized system \eqref{linearized_system}. Therefore:
		\begin{equation}
			\lim_{\varepsilon \to 0} \Phi_I(\varepsilon)= \int_0^T \left[\alpha \int_{\mathbb{R}^N} \xi(t,x)dx + 2\beta I(t)h(t)\right]dt + \gamma \int_{\mathbb{R}^N} \xi(T,x)dx = \langle L_I,h\rangle.
		\end{equation}
		This establishes that $J$ is Gâteaux differentiable at $I$ with derivative $J'(I) = L_I$, which complete the proof.
	\end{proof}
\subsection{Necessary conditions}
	To simplify the expression of the gradient of $J$, we introduce an adjoint
	system. This system is designed to eliminate the dependence on $\xi$ in the
	expression of $J^\prime(I)$. We introduce the following system:
	\begin{equation}\label{adjoint_system}
		\begin{cases}
			\vspace{0.2cm}
			\displaystyle	-\dfrac{\partial q}{\partial t}(t,x) -
			\mathrm{V}\left[w_p(t,x)\right] \cdot \nabla q(t,x) -
			\int_{\Omega}K(y,x)\left( p\mathrm{V}^{\prime}\left[w_p\right]\cdot\nabla
			q\right) (t,y)  d y  +  g(d,p) (t,x) q(t,x) = \alpha, \quad&(a)\\
			\vspace{0.2cm}
			-\dfrac{\partial r}{\partial t}(t,x) - \operatorname{div}(D(x)\nabla r(t,x)) +
		\partial_2\Gamma(I, d)(t,x)r(t,x) = -\partial_2 C(d,p)q(t,x) , &(b)\\
			q(T,x) = \gamma, \quad r(T,x) = 0.
		\end{cases}
	\end{equation}
	We first establish the following lemma: 
	\begin{lemma}[Well-posedness of the Adjoint System]
		Let Assumptions \textbf{(A1)-(A12)} hold. Then the adjoint system \eqref{adjoint_system}
		has a unique solution $(q,r)$ in $C(0,T; L^{\infty}(\mathbb{R}^N)) \times C(0,T; H^1(\mathbb{R}^N))$.
	\end{lemma}
	\begin{proof}
		Let $s = T-t$ and define:
		\begin{equation}
			\bar{q}(s,x) = q(T-s,x), \quad \bar{r}(s,x) = r(T-s,x).
		\end{equation}
		The transformed system becomes:
		\begin{equation}
			\left\{
			\begin{array}{ll}
				\vspace{0.15cm}
			\displaystyle	\dfrac{\partial \bar{q}}{\partial s} - \mathrm{V}[w_p]\cdot\nabla \bar{q} - \int_{\Omega}K(y,x)\left(p\mathrm{V}'[w_p]\cdot\nabla \bar{q}\right)(T-s,y)dy + g(d,p)\bar{q} = \alpha, &(a)\\
				\vspace{0.15cm}
				\dfrac{\partial \bar{r}}{\partial s} - \operatorname{div}(D(x)\nabla \bar{r}) + 
				\partial_2\Gamma(I, d)\bar{r} = -\partial_2C(d,p)\bar{q}, &(b)\\
				\bar{q}(0,x) = \gamma, \quad \bar{r}(0,x) = 0.
			\end{array}
			\right.
		\end{equation}
		By using assumptions \textbf{(A11)} and \textbf{(A12)}, the integration by parts in the nonlocal term leads to:
		\begin{equation}
				\int_{\Omega}K(y,x)\left(p\mathrm{V}'[w_p]\cdot\nabla \bar{q}\right)(T-s,y)dy = -\int_{\Omega}\bar{q}(s,y)\operatorname{div}_y\Bigl(p(y)\mathrm{V}'[w_p(y)]K(y,x)\Bigr)dy,
		\end{equation}
		where the boundary terms vanish since $K(y,x) = 0$ for $y \in \partial\Omega$ by \textbf{(A10)} and $\mathrm{V}'(0) = 0$ by \textbf{(A11)}.
	Therefore, $(\bar{q},\bar{r})$ satisfies:
	\begin{equation}\label{tr_adj}
		\left\{
		\begin{array}{ll}
			\displaystyle \dfrac{\partial \bar{q}}{\partial s} - \operatorname{div}\left( \mathrm{V}[w_p]\bar{q}\right)  + \int_{\Omega}\bar{q}(y)\operatorname{div}_y\left(p(y)\mathrm{V}'[w_p(y)]K(y,x)\right)dy = -\left(g(d,p) + \operatorname{div}(\mathrm{V}[w_p])\right)\bar{q} + \alpha, \\[0.3cm]
			\dfrac{\partial \bar{r}}{\partial s} - \operatorname{div}(D(x)\nabla \bar{r}) + 	\partial_2\Gamma(I, d)\bar{r} = -\partial_2C(d,p)\bar{q}, \\[0.3cm]
			\bar{q}(0,x) = \gamma, \quad \bar{r}(0,x) = 0.
		\end{array}
		\right.
	\end{equation}
		The transformed system has a similar structure as the linearized system \eqref{linearized_system}. Therefore, by the same argument as in Lemma~ \ref{lem:linearized_wellposed}, the system \eqref{tr_adj} has a unique solution $(\bar{q},\bar{r})$, which imply the existence and uniqueness of $(q,r)$ for the adjoint system \eqref{adjoint_system}.
	\end{proof} 
	Now we can state and prove the necessary optimality conditions:
	\begin{theorem}[Necessary Optimality Conditions]
		\label{thm:optimality_conditions}
		Let $I^* \in \mathcal{U}_{ad}$ be an optimal control for the problem with
		corresponding state $(p^*, d^*)$. Then there exists an adjoint state $(q,r)$
		satisfying the adjoint system \eqref{adjoint_system} such that:
		\begin{equation}
			\int_0^T \left( 2\beta I^*(t) +  \int_{\mathbb{R}^N} \partial_1\Gamma(I^*, d^*) r(t,x) dx\right)
			\left( v(t) - I^*(t)\right) dt \geq 0
		\end{equation}
		for all $v \in \mathcal{U}_{ad}$.
	\end{theorem}
	\begin{proof}
		Let $I^*$ be an optimal control. For any $v \in \mathcal{U}_{ad}$ and
		$\varepsilon \in (0,1)$, define $I^\varepsilon = I^* + \varepsilon(v - I^*)$.
		Note that $I^\varepsilon \in \mathcal{U}_{ad}$ due to the convexity of
		$\mathcal{U}_{ad}$.\\
		By the optimality of $I^*$, we have:
		\begin{equation}
			0 \leq \lim_{\varepsilon \to 0^+} \frac{J(I^\varepsilon) -
				J(I^*)}{\varepsilon} = \langle J'(I^*), v - I^* \rangle.
		\end{equation}
		From Theorem \ref{thm:cost_diff}, we have:
		\begin{equation}
			\langle J'(I^*), v - I^* \rangle = \int_0^T \left[\alpha \int_{\mathbb{R}^N}
			\xi(t,x)dx + 2 \beta I^*(t)(v(t) - I^*(t))\right]dt + \gamma
			\int_{\mathbb{R}^N} \xi(T,x)dx
		\end{equation}
		where $(\xi, \eta)$ solves the linearized system with $h = v - I^*$.\\
		We multiply the first equation of the linearized system
		(\ref{linearized_system}-a) by q, and the second equation (\ref{linearized_system}-b) by r, then integrate
		over space and time:
			\begin{align*}
			&\int_0^T \int_{\mathbb{R}^N} q\left(\frac{\partial \xi}{\partial t} +
			\operatorname{div}(\mathrm{V}[w_{p^*}]\xi) +
			\operatorname{div}(\mathrm{V}'[w_{p^*}](w_\xi)p^*) - g(d^{*},p^{*})\xi + \partial_1 C(d^{*},p^{*})\eta\right) dxdt = 0 ,\\
			&\int_0^T \int_{\mathbb{R}^N} r\left(\frac{\partial \eta}{\partial t} -
			\operatorname{div}(D(x)\nabla \eta) - \partial_{2}\Gamma(I,d)\eta-\partial_{1}\Gamma(I,d)(v -
			I^*)\right) dxdt = 0.
		\end{align*}
		After integration by parts we get for $\xi$ equation:
		\begin{align*}
			&\gamma\int_{\mathbb{R}^N} \xi(T,x)dx - \int_0^T \int_{\mathbb{R}^N}
			\frac{\partial q}{\partial t}\xi dxdt - \int_0^T \int_{\mathbb{R}^N} \nabla q
			\cdot (\mathrm{V}[w_{p^*}]\xi) dxdt \\
			&- \int_0^T \int_{\mathbb{R}^N} \nabla q \cdot (\mathrm{V}'[w_{p^*}]w_\xi p^*
			)dxdt \\
			&= \int_0^T \int_{\mathbb{R}^N} qF'(p^*)\xi dxdt - \int_0^T \int_{\mathbb{R}^N}
			q\partial_2 C\xi dxdt - \int_0^T \int_{\mathbb{R}^N}
			q\partial_1 C\eta dxdt.
		\end{align*}
		We have by using Fubini theorem:
		{\footnotesize
			\begin{align*}
				\int_0^T \int_{\mathbb{R}^N} \nabla q(t,x) \cdot
				(\mathrm{V}'[w_{p^*}(t,x)]w_\xi(t,x) p^*(t,x) )dxdt&=\int_0^T
				\int_{\mathbb{R}^N}\int_{\Omega}\nabla q(t,x) \cdot
				\mathrm{V}'[w_{p^*}(t,x)]k(x,y)\xi(t,y)dx dy dt \\
				&=\int_0^T \int_{\mathbb{R}^N}\xi(t,x)\int_{\Omega}
				\nabla q(t,y) \cdot (\mathrm{V}'[w_{p^*}(t,y)]k(y,x))dy dx dt.
			\end{align*}
		}
		Hence we get:
		{\footnotesize
			\begin{align*}
				\gamma\int_{\mathbb{R}^N} \xi(T,x)dx&+\int_0^T \int_{\mathbb{R}^N}\xi\left[ -
				\frac{\partial q}{\partial t}- \nabla q \cdot (\mathrm{V}[w_{p^*}]) 
				- \int_{\mathbb{R}^N} \nabla q \cdot (\mathrm{V}'[w_{p^*}]p^*k(x,\cdot))dx -
				g(d^{*},p^{*})\; q\right]  dxdt \\&- \int_0^T
				\int_{\mathbb{R}^N} q\partial_1 C(d^{*},p^{*})\eta dxdt=0.
			\end{align*}
		}
		By substituting \eqref{adjoint_system} we get:
		\begin{equation}\label{eq1}
			\gamma\int_{\mathbb{R}^N} \xi(T,x)dx+\alpha \int_0^T  \int_{\mathbb{R}^N}
		\xi(t,x)dx dt+ \int_0^T \int_{\mathbb{R}^N} q\partial_{1} C(d^{*},p^{*})\eta dxdt=0.
		\end{equation} 
		In the other hand, we have for the second equation after integrating by parts: 
		\begin{align*}
			&\int_{\mathbb{R}^N} r(T,x)\eta(T,x)dx - \int_{\mathbb{R}^N} r(x,0)\eta(x,0)dx
			- \int_0^T \int_{\mathbb{R}^N} \frac{\partial r}{\partial t}\eta dxdt \\
			&+ \int_0^T \int_{\mathbb{R}^N} \nabla r \cdot (D(x)\nabla \eta) dxdt -
			\int_0^T \int_{\mathbb{R}^N} r\partial_2\Gamma(I^*, d^*)\eta dxdt \\
			&= \int_0^T \int_{\mathbb{R}^N} r\partial_1\Gamma(I^*, d^*)(v - I^*) dxdt,
		\end{align*}
		after substituting \eqref{adjoint_system} we get:
		\begin{equation}\label{eq2}
			\int_0^T \int_{\mathbb{R}^N} q\partial_1 C(d^{*},p^{*})\eta dxdt=-\int_0^T
			\int_{\mathbb{R}^N} r\partial_1\Gamma(I^*, d^*)(v - I^*) dxdt.
		\end{equation}
		By substituting \eqref{eq1} and \eqref{eq2} back into \eqref{dJ} we obtain:
		\begin{equation}
			\int_0^T \Bigl(2\beta I^*(t) + \int_{\mathbb{R}^N} \partial_1\Gamma(I^*, d^*)r(x,t)dx\Bigr)\Bigl(v(t) -
			I^*(t)\Bigr)dt \geq 0
		\end{equation}
		for all $v \in \mathcal{U}_{ad}$, which completes the proof.
	\end{proof}
		\begin{cor}[Characterization of Optimal Control]
		Let Assumptions \textbf{(A1)-(A12)} hold and let $I^* \in \mathcal{U}_{ad}$ be
		an optimal control with corresponding adjoint state $r$. Then:
		\begin{equation}\label{I_opt}
			I^*(t) = \max\left\lbrace 0,\min\left\lbrace
			M,-\frac{1}{2\beta}\int_{\mathbb{R}^N} \partial_1\Gamma(I^*, d^*)r(t,x)dx\right\rbrace
			\right\rbrace  \quad \text{a.e. } t \in [0,T].
		\end{equation}
	\end{cor}
	\begin{proof}
		Let $I^*$ be an optimal control for problem (\eqref{OC_problem}). From Theorem~\ref{thm:optimality_conditions}, we have the following variational inequality:
		\begin{equation}
			\int_0^T \left(2\beta I^*(t) + \int_{\mathbb{R}^N} \partial_1\Gamma(I^*, d^*)r(t,x)dx\right)(v(t) - I^*(t))dt \geq 0, \quad \forall v \in \mathcal{U}_{ad}.
		\end{equation}
		
		We introduce the switching function $\displaystyle \psi(t) := 2\beta I^*(t) + \int_{\mathbb{R}^N} \partial_1\Gamma(I^*, d^*)r(t,x)dx$. The variational inequality can thus be rewritten as:
		\begin{equation}
			\int_0^T \psi(t)(v(t) - I^*(t))dt \geq 0, \quad \forall v \in \mathcal{U}_{ad}.
		\end{equation}
		Let us consider first the case where $\psi(t) > 0$ for some $t \in [0,T]$. Let suppose, by contradiction, that there exists a set of positive measure where $\psi(t) > 0$ and $I^*(t) > 0$. For sufficiently small $\varepsilon > 0$, we can define an admissible control
		$v_\varepsilon(t):=
				\max\{0, I^*(t) - \varepsilon\}.$\\
		For this choice of $v_\varepsilon$, we have $v_\varepsilon(t) - I^*(t) = -\varepsilon < 0$, which leads to
	$\displaystyle
			\int_0^T \psi(t)(v_\varepsilon(t) - I^*(t))dt < 0
		$,
which contradicts the optimality condition, thus $I^*(t) = 0$ whenever $\psi(t) > 0$.
		
		For the case where $\psi(t) < 0$, we can apply a similar argument. Let suppose there exists a set of positive measure where $\psi(t) < 0$ and $I^*(t) < M_{tol}$. We define $v_\varepsilon(t) =
				\min\{M_{tol}, I^*(t) + \varepsilon\}$.
		This choice leads to a similar contradiction, establishing that $I^*(t) = M_{tol}$ whenever $\psi(t) < 0$.
		
		When $\psi(t) = 0$, we have directly:
		\begin{equation*}
			2\beta I^*(t) + \int_{\mathbb{R}^N} \partial_1\Gamma(I^*, d^*)r(t,x)dx = 0.
		\end{equation*}
		Therefore:
		\begin{equation*}
			I^*(t) = -\frac{1}{2\beta}\int_{\mathbb{R}^N} \partial_1\Gamma(I^*, d^*)r(t,x)dx.
		\end{equation*}
			The constraints $0 \leq I^*(t) \leq M_{tol}$ must hold for all $t \in [0,T]$. Combining these constraints with our previous analysis, we obtain the complete characterization:
		\begin{equation}
			I^*(t) = \max\left\{0, \min\left\{M_{tol}, -\frac{1}{2\beta}\int_{\mathbb{R}^N} \partial_1\Gamma(I^*, d^*)r(t,x)dx\right\}\right\} \quad \text{a.e. } t \in [0,T].
		\end{equation}
	\end{proof}
	\section{Numerical Simulations}\label{sec:5}
	
	For numerical implementation of the optimal control strategy characterized by equation \eqref{I_opt}, we formulate a computational approach based on the variational structure of the control problem and the adjoint system derived in Section~\ref{sec4}. The iterative scheme is  as follows:
	\begin{center}
		\begin{itemize}
			\item At iteration $k$, given the current control approximation $I^k \in U_{ad}$.
			\begin{itemize}
				\item Compute the state variables $(p^k, d^k)$ by solving the coupled system \eqref{system1}.
				\item Compute the adjoint variables $(q^k, r^k)$ through backward integration of system \eqref{adjoint_system}.
				\item Evaluate the gradient of the cost functional $\nabla J(I^k)$.
			\end{itemize}
		\item Update $I^{k+1} = \Pi_{U_{ad}}\left(I^k - \alpha_k \nabla J(I^k)\right)$, where $\Pi_{U_{ad}}$ denotes the projection operator onto the admissible set $U_{ad}$ and $\alpha_k > 0$ represents the step size.
		\end{itemize}
	\end{center}
 While this iterative procedure provides a practical realization of the theoretical framework established in preceding sections, a formal convergence analysis of the numerical scheme lies beyond the scope of the present work. The numerical simulations were implemented using Python $3.13$ with NumPy library.
	
	\subsection{Discretization and Parameter Configuration}
	
	The computational domain $\Omega = [0,1]$ was discretized with a uniform spatial grid of $N_x = 200$ points, while the time domain $[0,T]$ with $T = 1$ was discretized according to the Courant-Friedrichs-Lewy (CFL) condition with $\text{CFL} = 0.025$. Let $\{x_i\}_{i=1}^{N_x}$ and $\{t_n\}_{n=1}^{N_t}$ denote the spatial and temporal discretization points, respectively. The initial condition $p_0(x)$ was specified as a normalized Gaussian function centered at $x = 0.5$ with standard deviation $0.1$,	while $d_0(x) \equiv 0$ throughout the domain. For the abstract functions introduced in system \eqref{system1}, we specify the nonlocal velocity field $\mathrm{V}[w_p] = \kappa \cdot w_p(t,x)$, the logistic growth function $F(p) = r \cdot p \cdot (1-p/K_p)$, the drug-induced cell death term $C(d,p) = \delta \cdot d \cdot p$, and the drug exchange and clearance mechanism $\Gamma(I,d) = \Gamma \cdot (I-d) - \lambda \cdot d$ (see \cite{trachette1999mathematical}). Where $w_p(t, x) = \int_{\Omega} K(x, y)p(t, y) \, dy$ represents the nonlocal term with $K$ defined as a Gaussian kernel with standard deviation $\sigma$.
	
	For the spatial discretization, the advection term in the equation \eqref{system1}-(b) was approximated using a Lax-Friedrichs scheme, the diffusion term in the equation \eqref{system1}-(c) was discretized using a standard second-order central difference approximation, and the nonlocal convolution integral \eqref{nonlocalterm} was computed using the Fast Fourier Transform (FFT). The resulting semi-discrete system was integrated in time using an explicit Euler method with time step constrained by the CFL condition.
	
	\subsection{Results and Analysis}
The parameter values used in simulations were: $r = 1.2$, $K_p = 1.0$, $\delta = 0.8$, $\kappa = 0.05$, $\sigma = 0.01$, $D = 0.05$, $\Gamma = 2.5$, and $\lambda = 0.3$. For the cost functional defined in equation \eqref{Jcost}, the weights were $\alpha = 1.0$, $\beta = 0.1$, and $\gamma = 1.0$, with the maximum tolerable drug concentration constrained to $M_{\text{tol}} = 4.0$.
		\begin{figure}[H]
		\centering
		\includegraphics[width=0.5\textwidth]{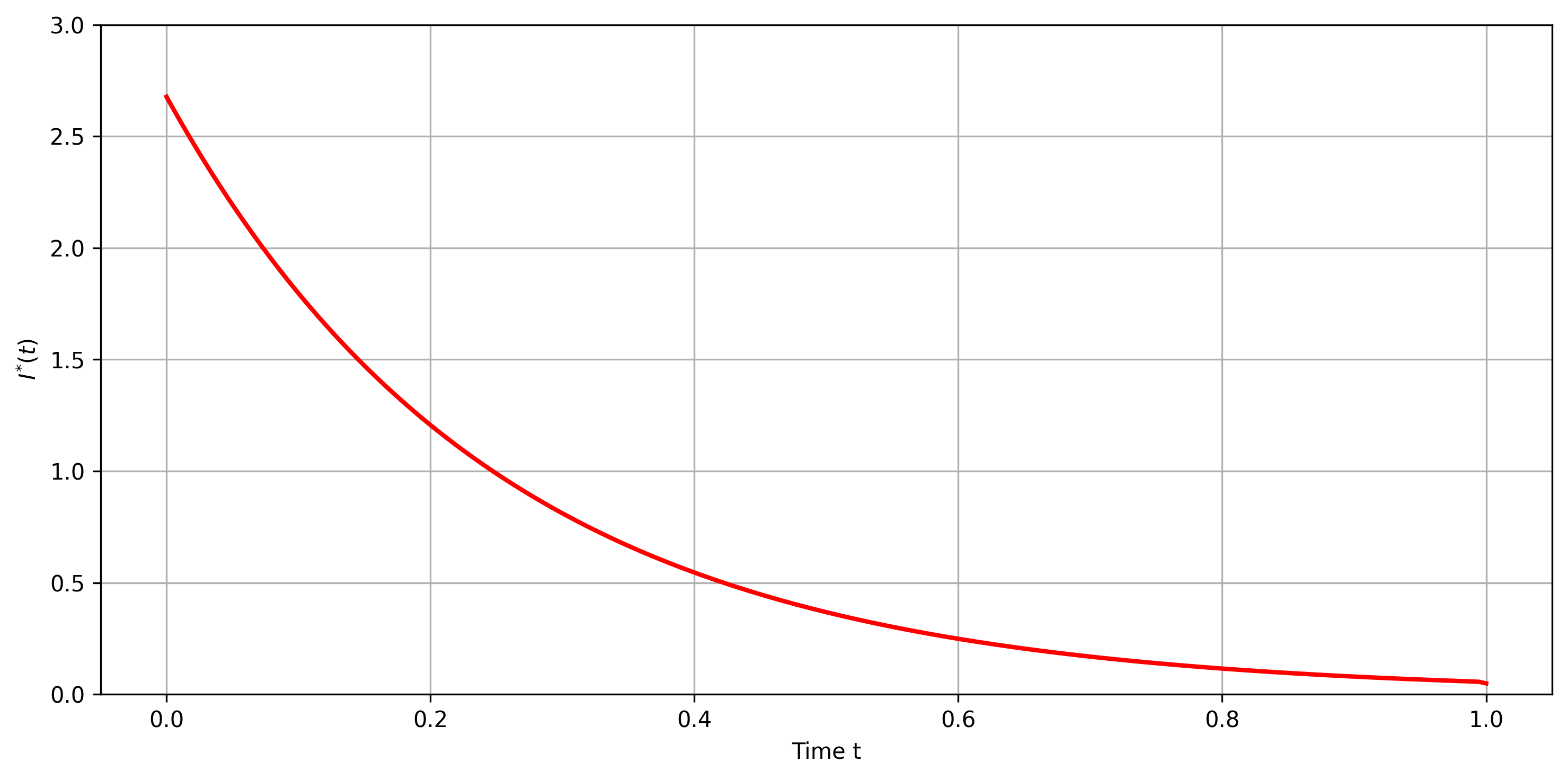}
		\caption{Optimal control profile $I^*(t)$ over the treatment interval $[0,T]$.}
		\label{fig:optimal-control}
	\end{figure}
	The computed optimal control $I^*(t)$ is a monotonically decreasing function over the treatment interval, as shown in Figure \ref{fig:optimal-control}. This decay is pronounced during the initial phase $t \in [0,0.3]$, followed by a more gradual reduction, suggesting that the optimal strategy entails substantial initial drug administration followed by progressive tapering. The asymptotic approach of $I^*(t)$
	toward zero as $t \rightarrow T$ reflects the balance between tumor reduction efficacy and control cost in the objective functional.

	The evolution of tumor density under optimal control shows a differences compared to the uncontrolled case, as illustrated in Figure \ref{fig:tumor-3d} and \ref{fig:tumor-contour}. Let $p_c(t,x)$ and $p_u(t,x)$ denote the tumor density functions under controlled and uncontrolled conditions, respectively. The relation $\sup_{x \in \Omega} p_c(t,x) < \sup_{x \in \Omega} p_u(t,x)$
	holds for all $t \in (0,T]$. The spatial profiles at discrete time points (Figure \ref{fig:tumor-profiles}) reveal progressive divergence between the controlled and uncontrolled distributions, having in approximately $27\%$ reduction in peak tumor density by $t=0.9$.
	\begin{figure}[H]
	\centering
	\begin{subfigure}{0.48\textwidth}
		\includegraphics[width=0.8\textwidth]{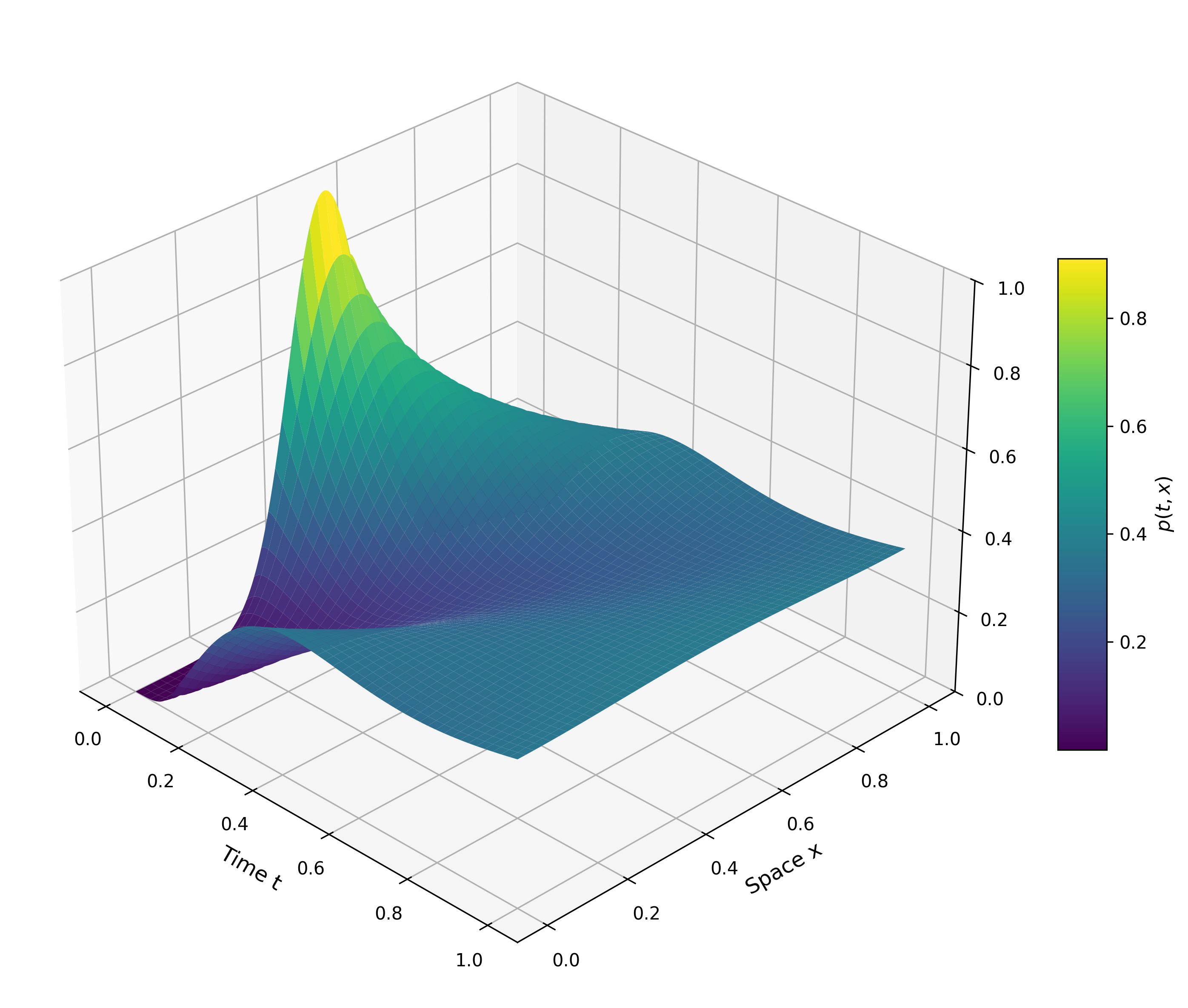}
		\caption{With optimal control}
	\end{subfigure}
	\hfill
	\begin{subfigure}{0.48\textwidth}
		\includegraphics[width=0.8\textwidth]{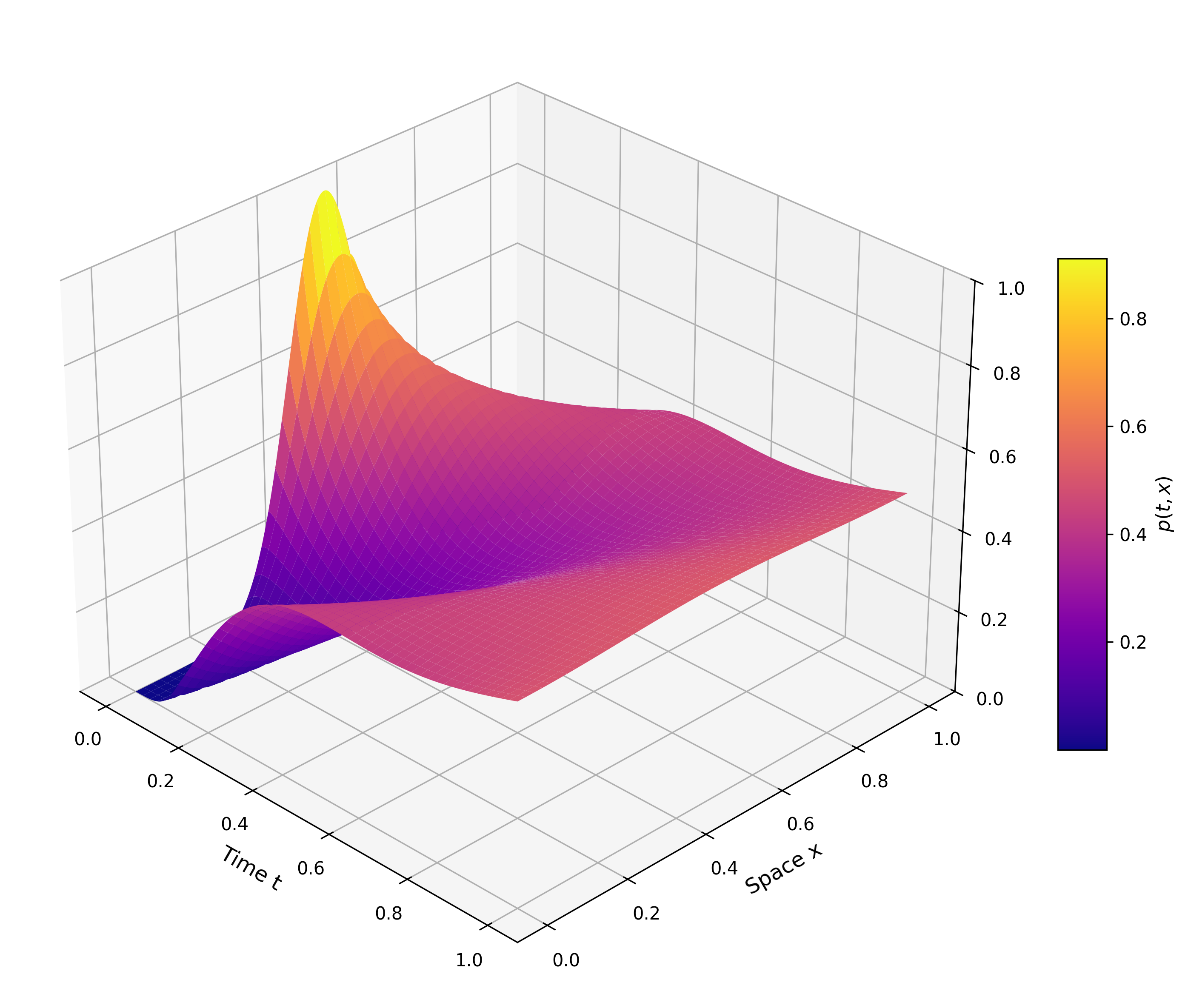}
		\caption{Without control}
	\end{subfigure}
	\caption{Three-dimensional visualization of tumor density evolution $p(t,x)$.}
	\label{fig:tumor-3d}
\end{figure}
\begin{figure}[H]
	\centering
	\begin{subfigure}{0.48\textwidth}
		\includegraphics[width=0.8\textwidth]{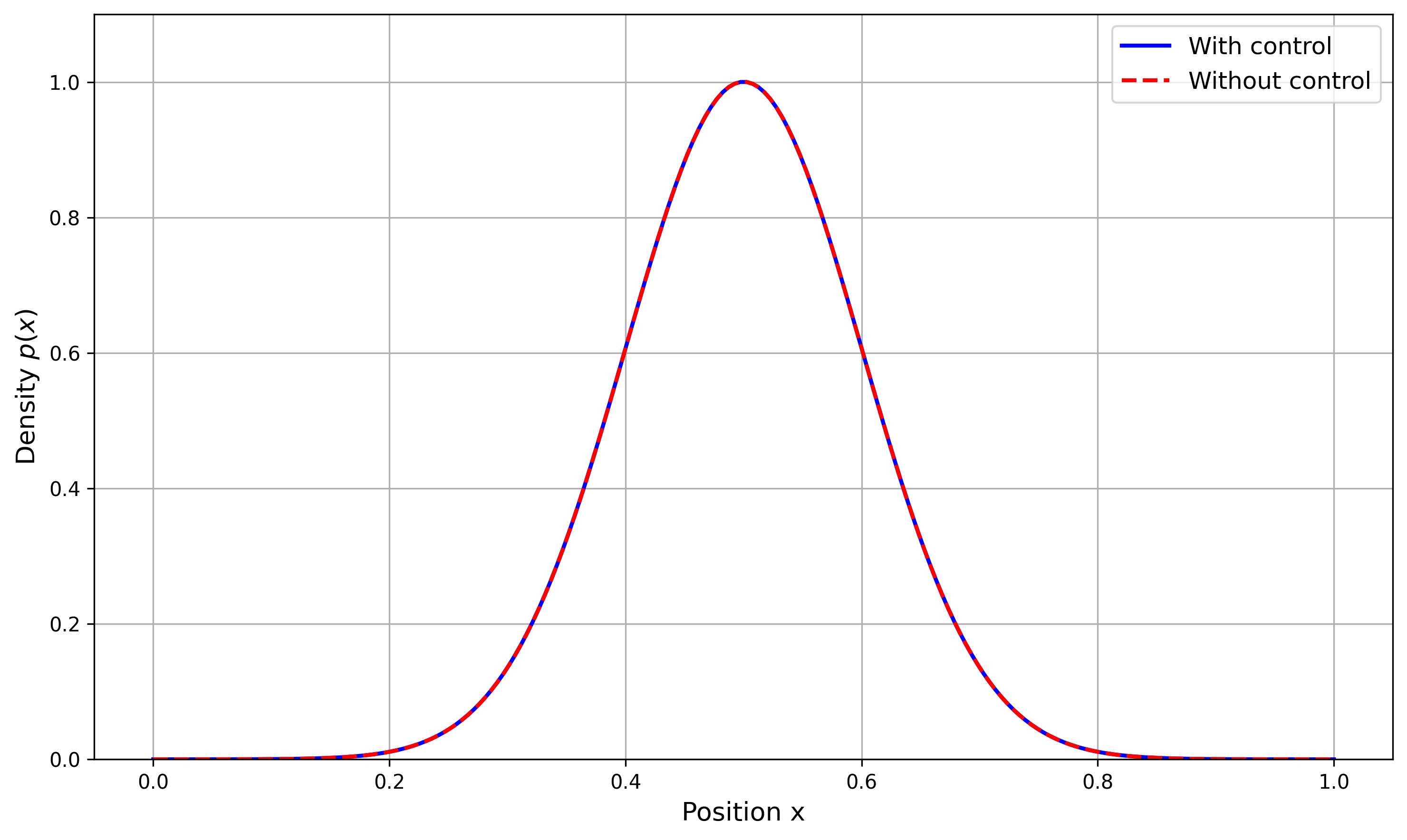}
		\caption{$t=0$ (initial distribution)}
	\end{subfigure}
	\hfill
	\begin{subfigure}{0.48\textwidth}
		\includegraphics[width=0.8\textwidth]{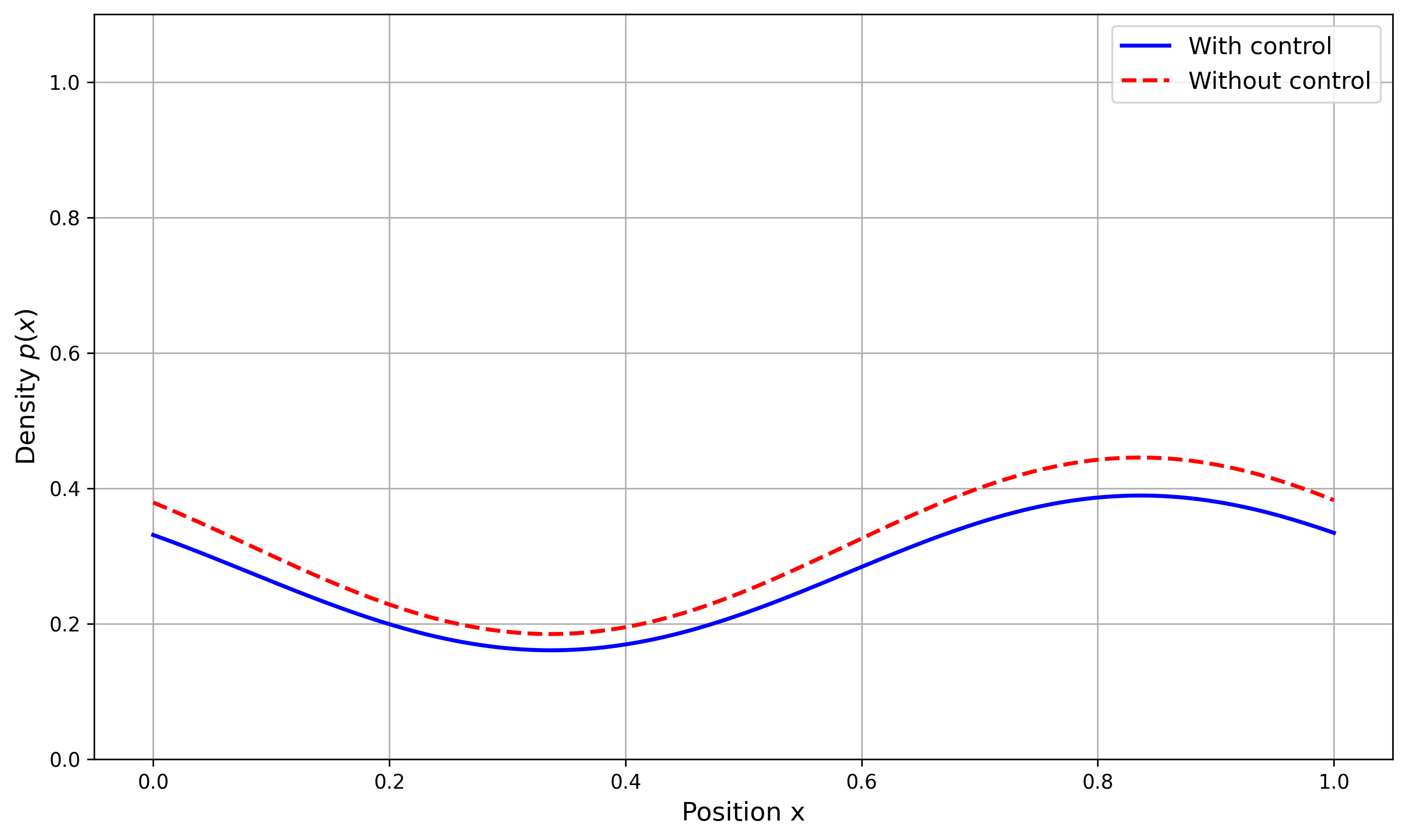}
		\caption{$t=0.3$}
	\end{subfigure}
	
	\vspace{0.25cm}
	
	\begin{subfigure}{0.48\textwidth}
		\includegraphics[width=0.8\textwidth]{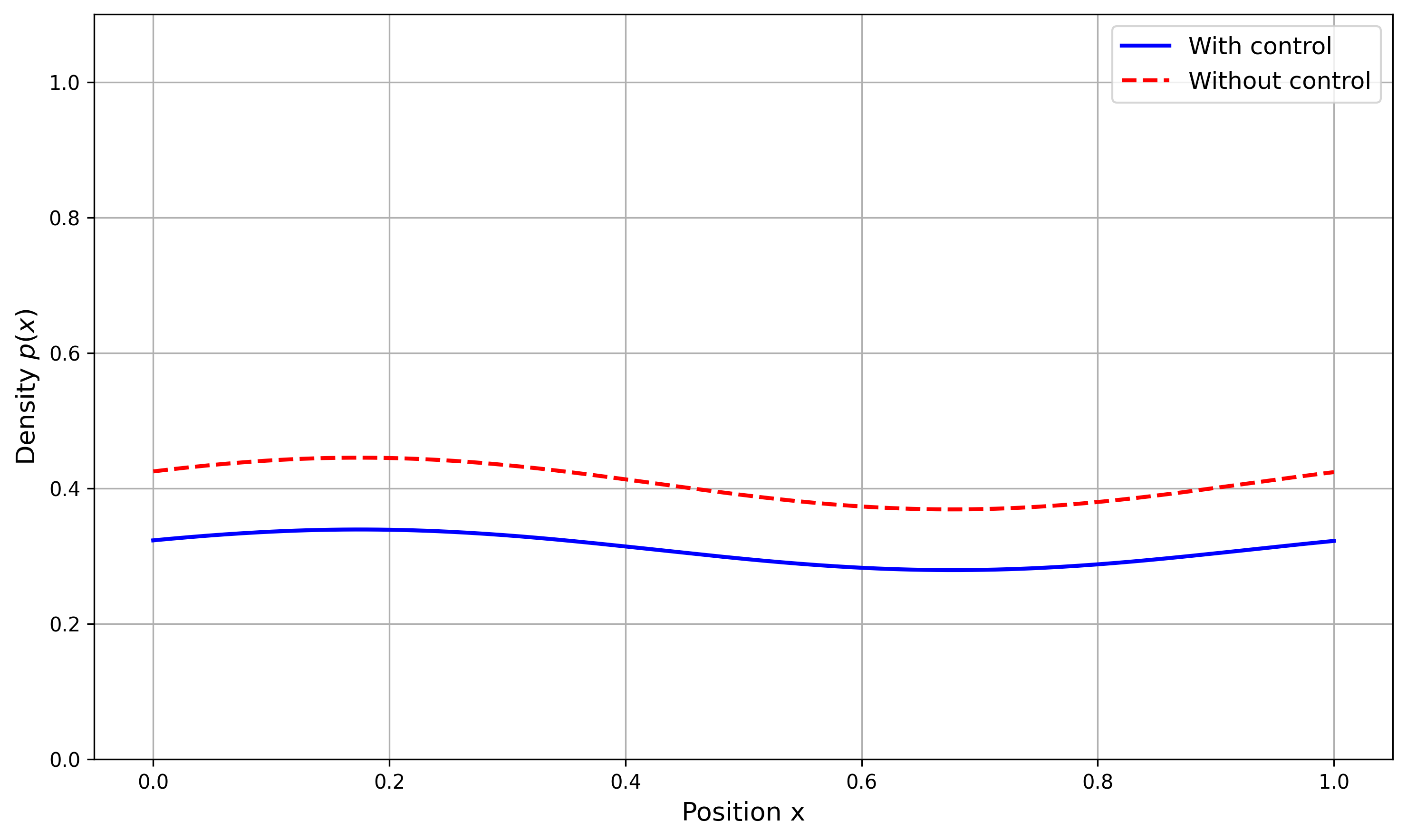}
		\caption{$t=0.7$}
	\end{subfigure}
	\hfill
	\begin{subfigure}{0.48\textwidth}
		\includegraphics[width=0.8\textwidth]{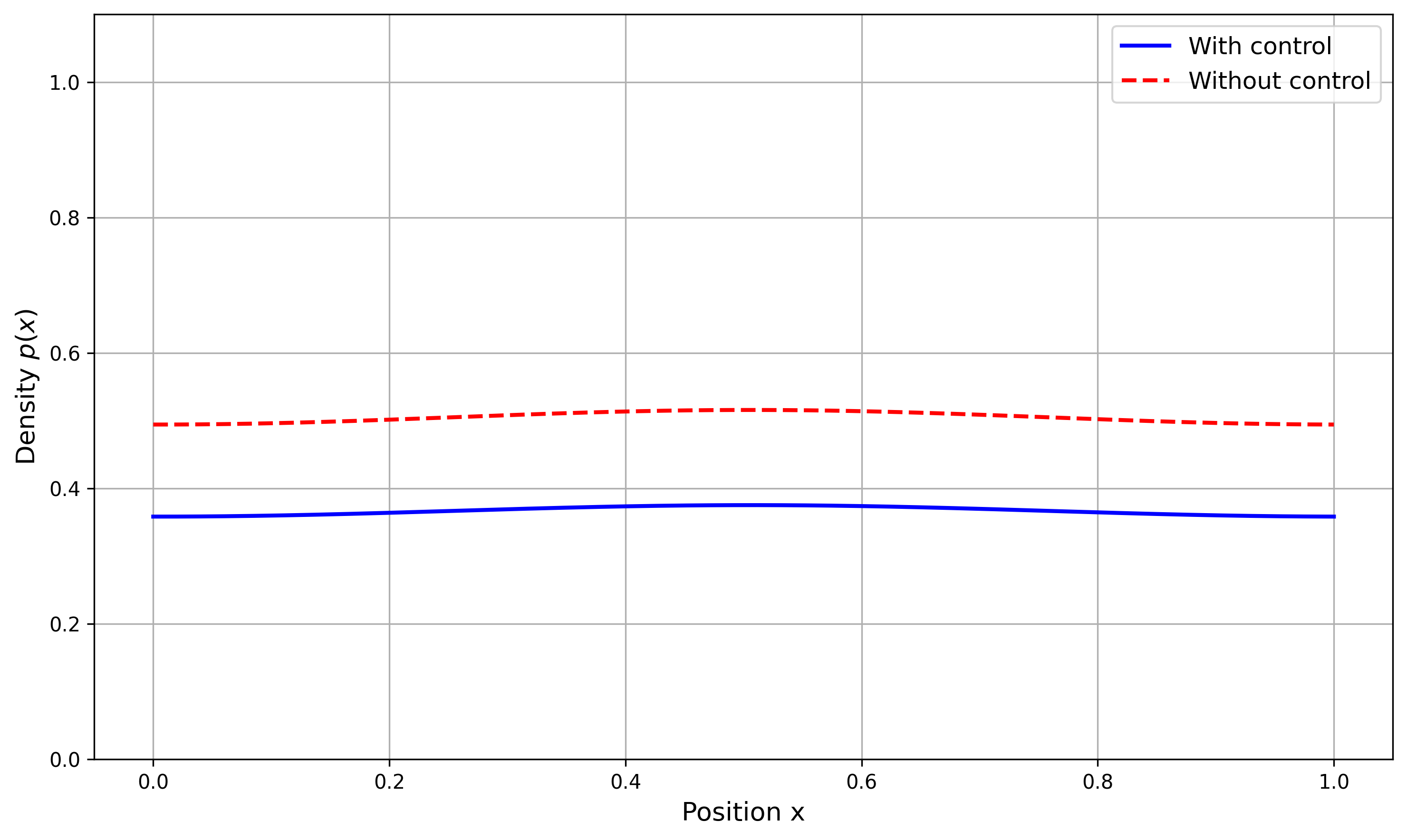}
		\caption{$t=0.9$}
	\end{subfigure}
	\caption{Tumor density profiles $p(t,x)$ at selected time points comparing controlled (solid blue) and uncontrolled (dashed red) cases.}
	\label{fig:tumor-profiles}
\end{figure}
\begin{figure}[H]
	\centering
	\begin{subfigure}{0.48\textwidth}
		\includegraphics[width=0.8\textwidth]{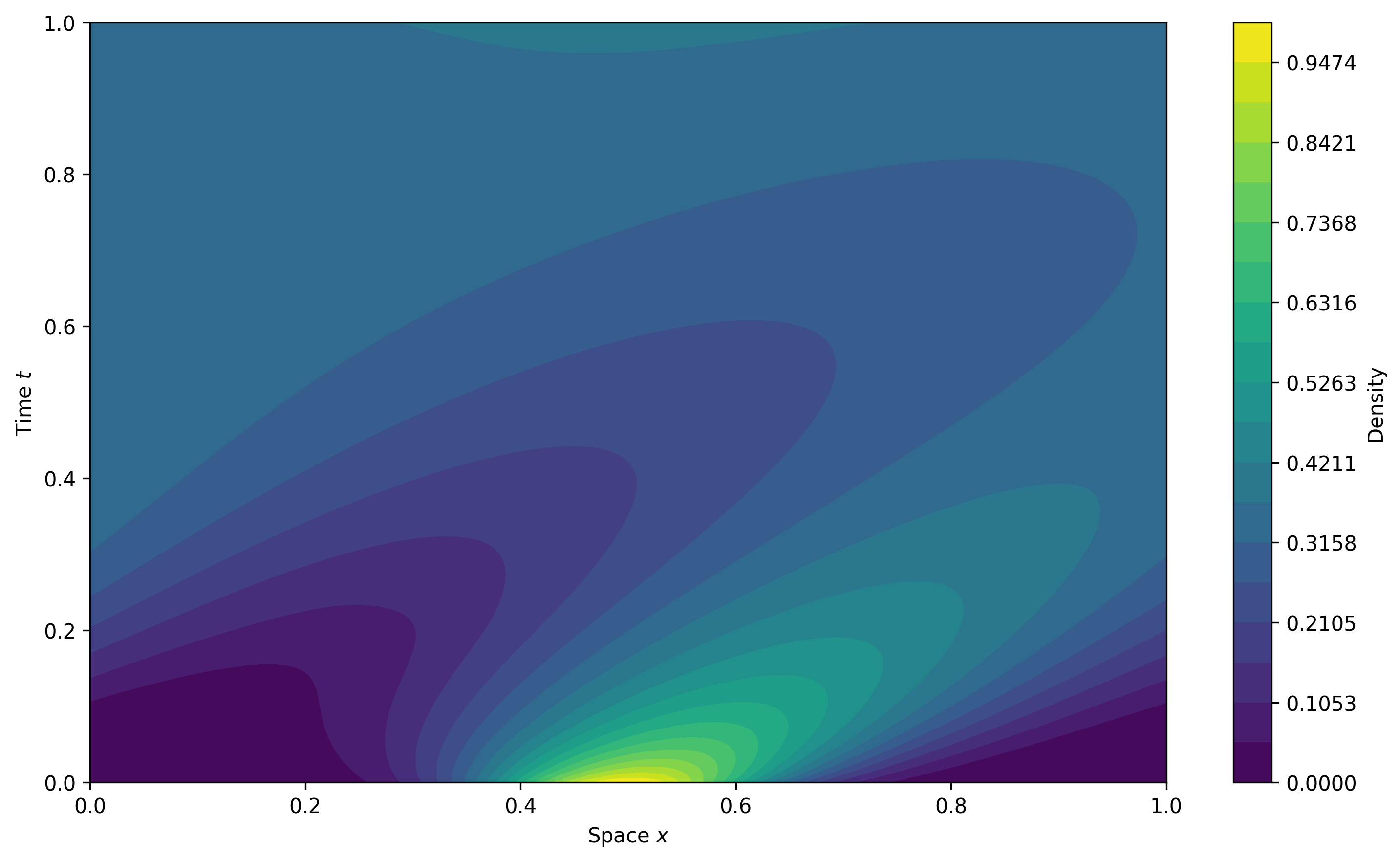}
		\caption{With optimal control}
	\end{subfigure}
	\hfill
	\begin{subfigure}{0.48\textwidth}
		\includegraphics[width=0.8\textwidth]{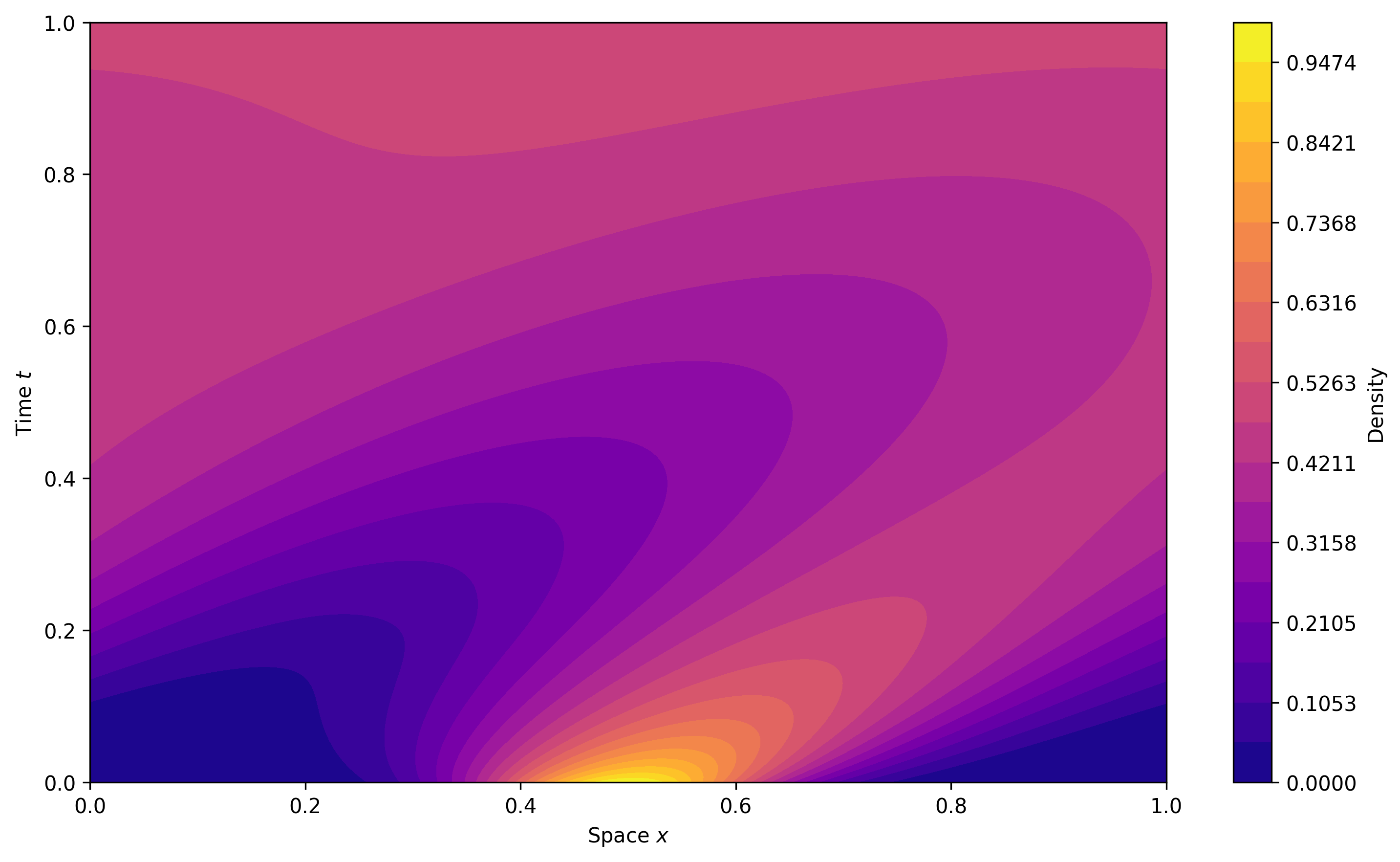}
		\caption{Without control}
	\end{subfigure}
	\caption{Contour plots of tumor density $p(t,x)$ in the $(x,t)$ plane.}
	\label{fig:tumor-contour}
\end{figure}

	The distribution of drug concentration $d(t,x)$
(Figure \ref{fig:drug-3d}) shows that this concentration is reaching a maximum value at $t \approx T/2$ despite the monotonically decreasing control input. This phase lag phenomenon emerges from the interaction between the source term and the diffusion operator, with the spatial homogeneity of $d(t,x)$ in contrast to the heterogeneous distribution of tumor density.
\begin{figure}[H]
	\centering
	\begin{subfigure}{0.48\textwidth}
		\includegraphics[width=0.8\textwidth]{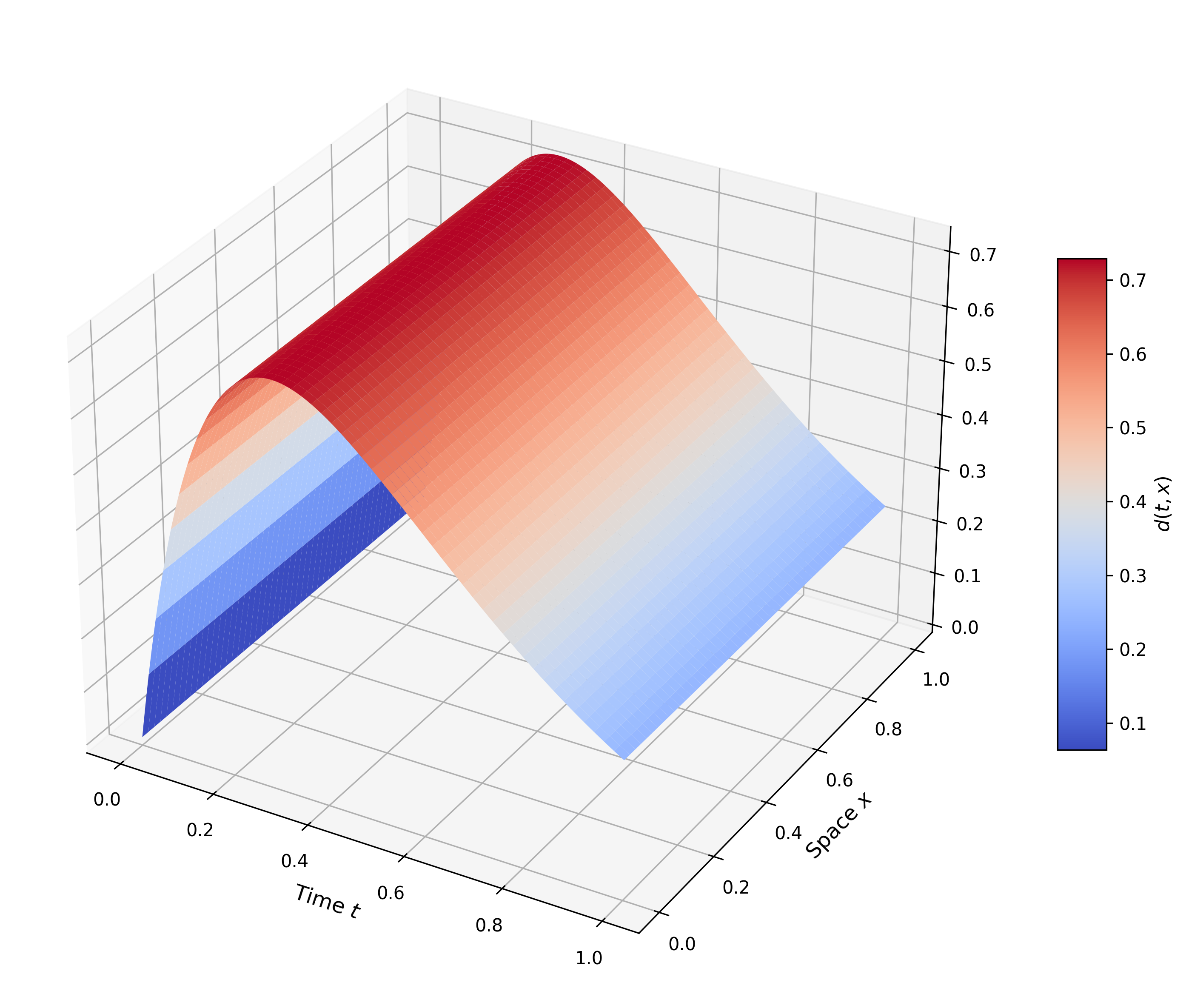}
		\caption{With optimal control}
	\end{subfigure}
	\hfill
	\begin{subfigure}{0.48\textwidth}
		\includegraphics[width=0.8\textwidth]{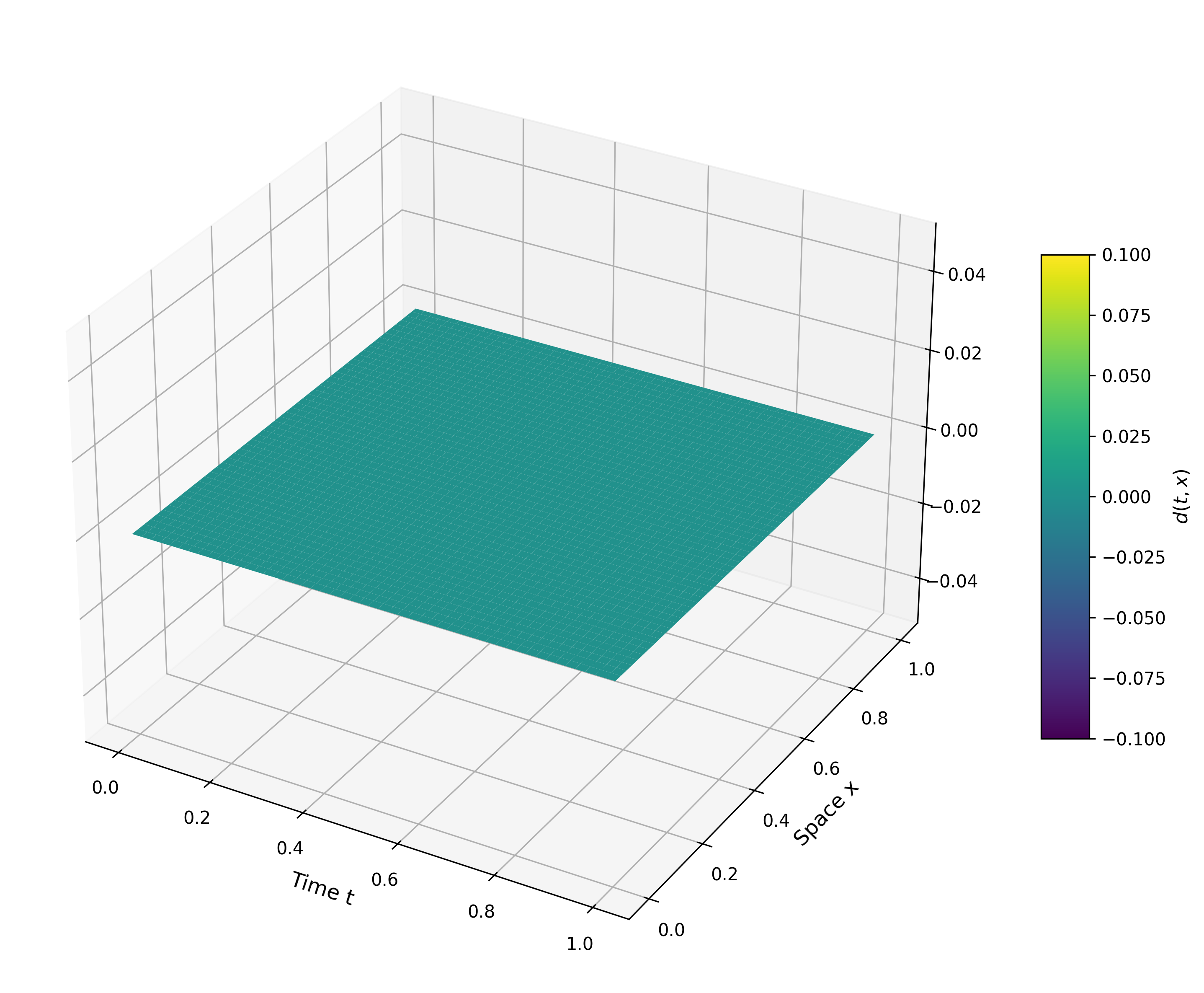}
		\caption{Without control}
	\end{subfigure}
	\caption{Three-dimensional visualization of drug concentration $d(t,x)$. }
	\label{fig:drug-3d}
\end{figure}
	To quantify the efficacy of the optimal control strategy, we define functionals $M_c(t) = \int_{\Omega} p_c(t,x) \, dx$
	and $M_u(t) = \int_{\Omega} p_u(t,x) \, dx$, representing the tumor mass under optimal control and without control, respectively. The relative improvement metric $ \Delta(t) = \frac{M_u(t) - M_c(t)}{M_u(t)} \times 100\%$
shows an increasing behavior for $t \in [0,T]$, with $\lim_{t \rightarrow T} \Delta(t) \approx 27\%$ (Figure \ref{fig:relative-improvement}), which quantifies the reduction achieved through the implementation of optimal control.
\begin{figure}[H]
	\centering
	\includegraphics[width=0.5\textwidth]{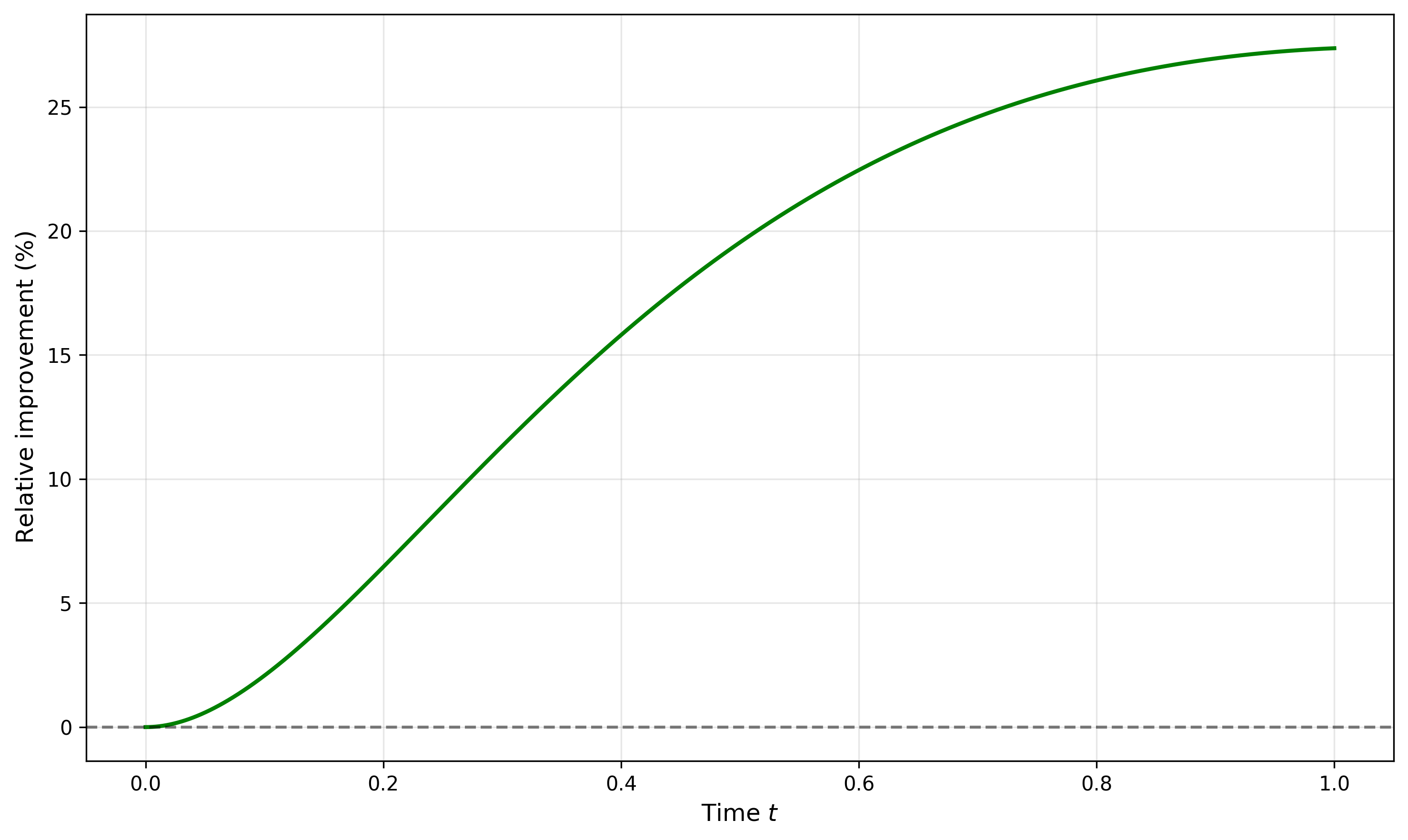}
	\caption{Relative improvement $\Delta(t)$ in tumor reduction due to optimal control intervention over time.}
	\label{fig:relative-improvement}
\end{figure}

\begin{figure}[H]
	\centering
	\includegraphics[width=0.5\textwidth]{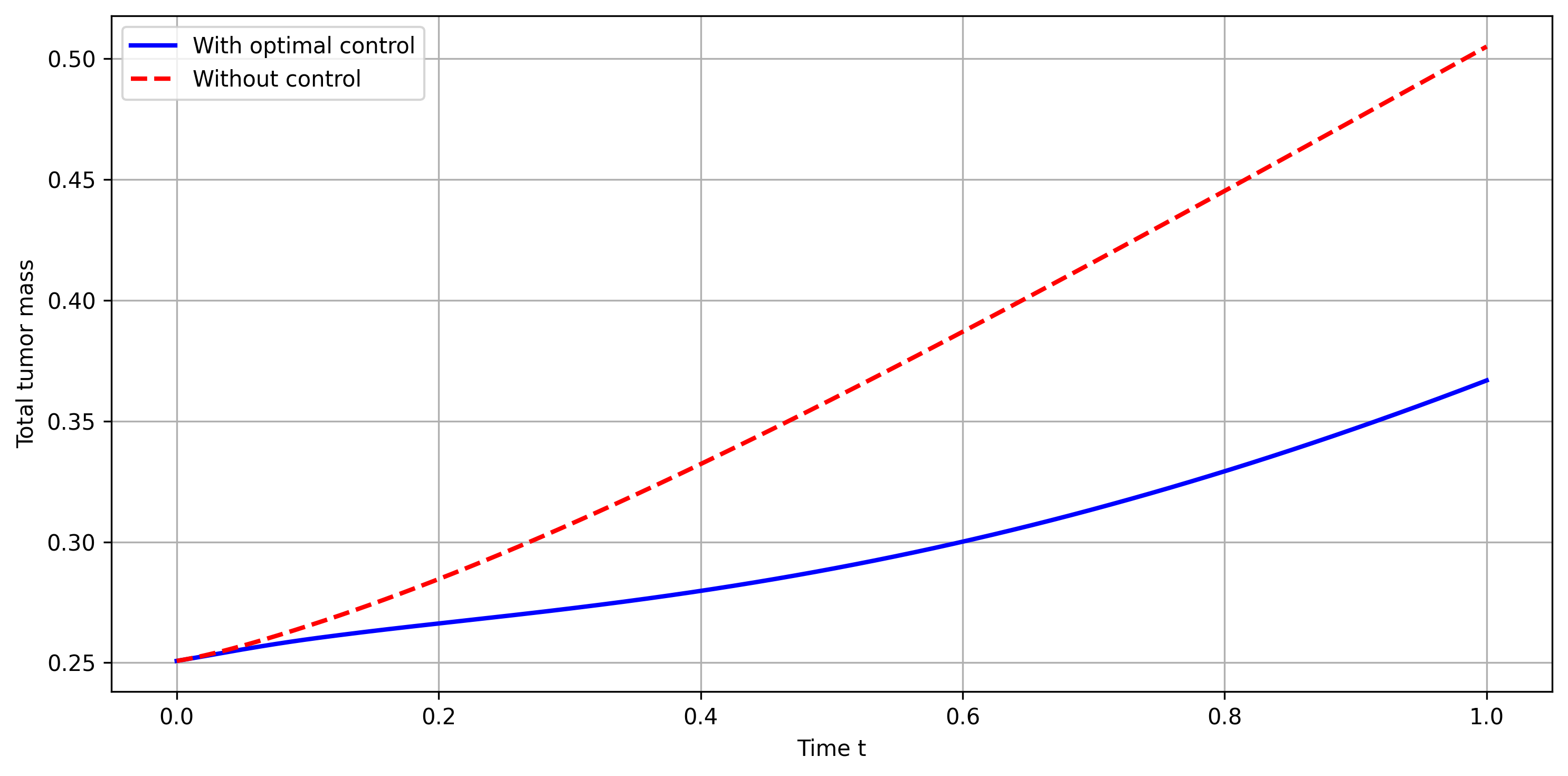}
	\caption{Comparison of tumor mass evolution with optimal control (solid blue) and without control (dashed red).}
	\label{fig:mass-comparison}
\end{figure}
	The numerical optimization procedure shows stable convergence characteristics, shown by the monotonic decrease in both the cost functional $J(I^k)$ and the gradient norm $\|\nabla J(I^k)\|_{L^2(0,T)}$ (Figure \ref{fig:convergence}). The algorithm was executed for a maximum of 50 iterations with the convergence criteria $|J(I^{k+1}) - J(I^k)|/|J(I^k)| < 5 \times 10^{-5}$.
\begin{figure}[H]
	\centering
	\begin{subfigure}{0.48\textwidth}
		\includegraphics[width=0.8\textwidth]{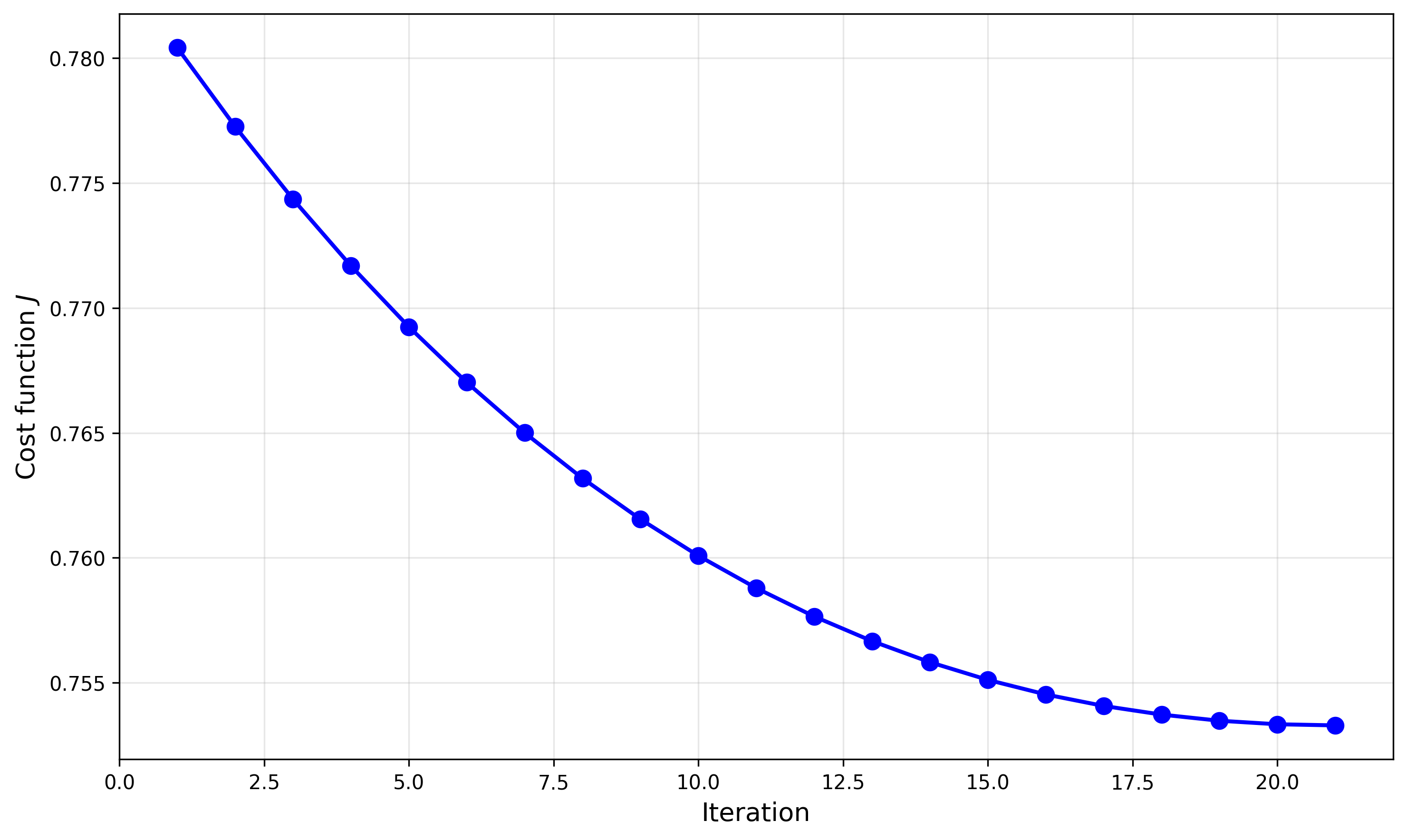}
		\caption{Cost function convergence}
	\end{subfigure}
	\hfill
	\begin{subfigure}{0.48\textwidth}
		\includegraphics[width=0.8\textwidth]{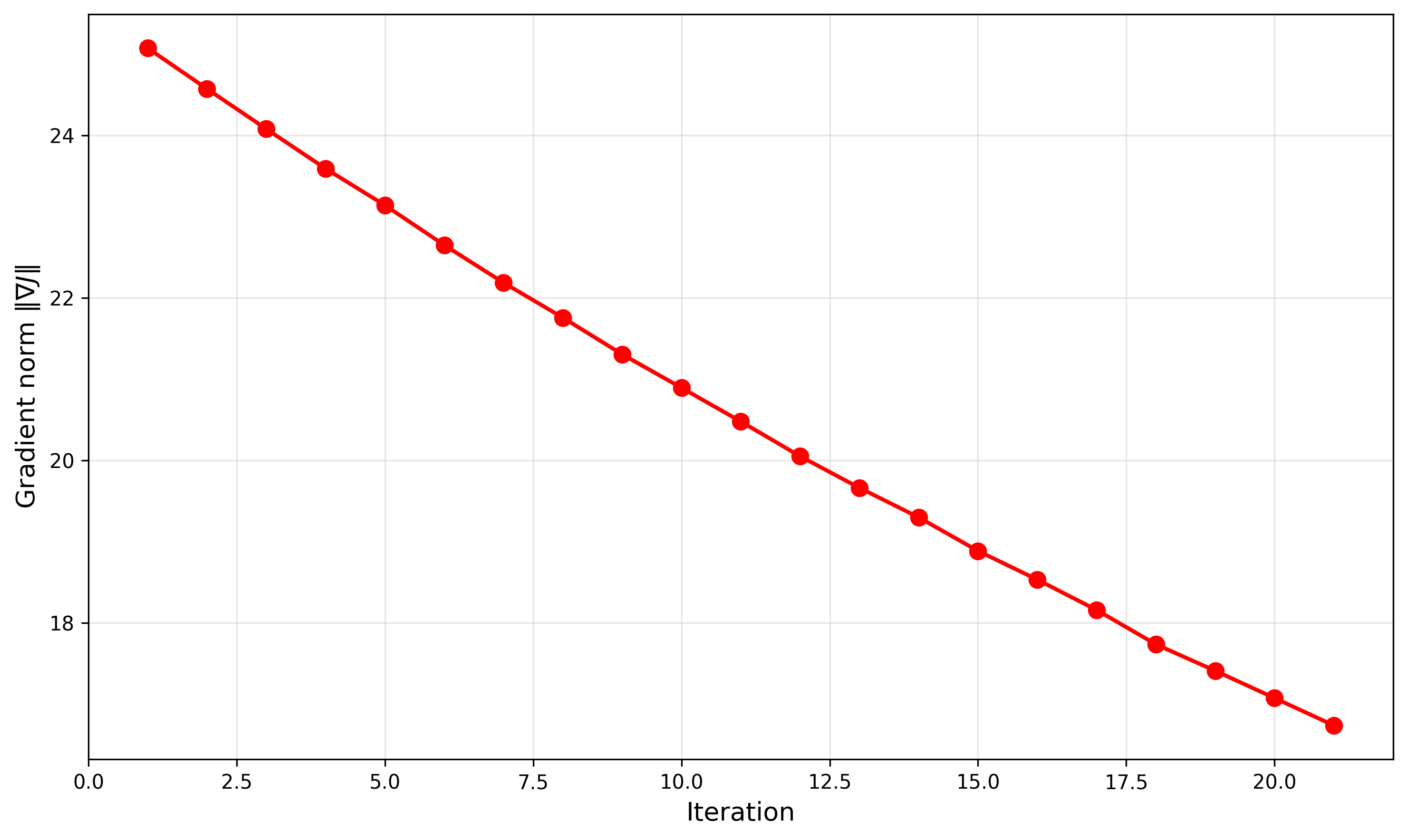}
		\caption{Gradient norm reduction}
	\end{subfigure}
	\caption{Evolution of optimization metrics across iterations: (a) monotonic decrease in cost functional $J(I^k)$ and (b) corresponding reduction in gradient norm $\|\nabla J(I^k)\|_{L^2(0,T)}$.}
	\label{fig:convergence}
\end{figure}
	
\begin{figure}[H]
	\centering
	\includegraphics[width=0.5\textwidth]{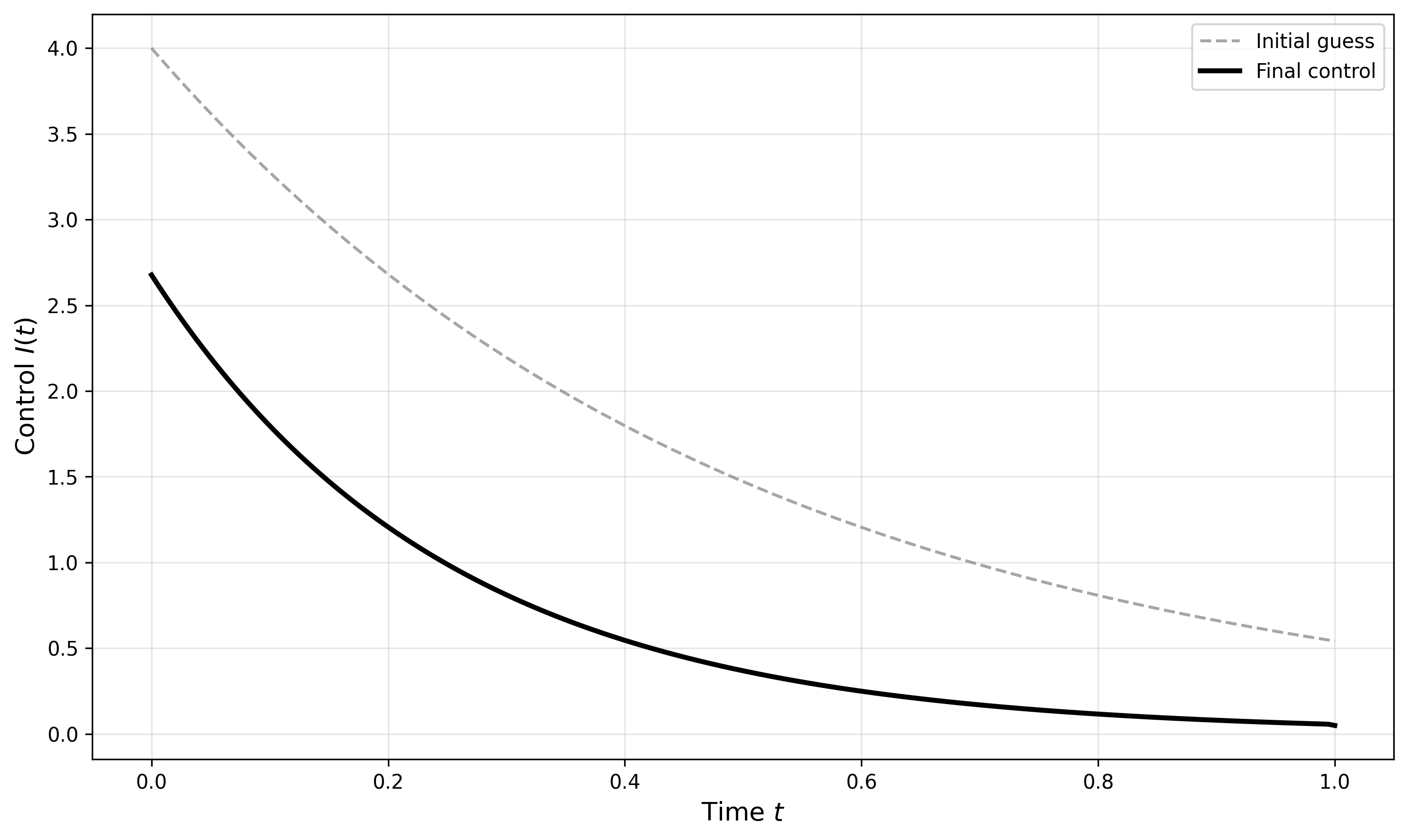}
	\caption{Evolution of control profile from initial guess (dashed gray) to optimal solution (solid black).}
	\label{fig:control-evolution}
\end{figure}
The initial control approximation, an exponential decay profile
$I^0(t) = M_{tol}e^{-4t}$, was chosen to align with the decreasing structure predicted by the necessary optimality conditions in equation \eqref{I_opt} while starting at the maximum tolerable dose. As illustrated in Figure \ref{fig:control-evolution}, this initial guess was progressively refined through optimization, converging to a profile with a steeper initial decrease and a smoother decrease.
\section*{Conclusion}
The development of effective cancer treatment strategies remains a critical challenge in modern medicine, requiring sophisticated mathematical frameworks that can capture the complex dynamics of tumor growth while providing practical guidance for therapy optimization. In this work, we have presented a mathematical analysis of a nonlocal tumor growth model incorporating chemotherapy, making some contributions to the theoretical understanding of optimal drug delivery.

Our analysis began by establishing a theoretical framework for studying nonlocal tumor growth models with chemotherapy, proving the existence and uniqueness of a solution under some assumptions. The incorporation of nonlocal terms, representing long-range cellular interactions, provides a more realistic model of tumor dynamics than traditional local approaches. We then developed an optimal control theory for this system, deriving the optimality conditions through a novel adjoint formulation that accounts for the nonlocal effects.

We obtained an explicit characterization of the optimal control strategy that provides practical guidance for drug dosing protocols. Our analysis reveals that optimal treatment schedules follow a precise structure determined by the interplay between tumor dynamics and drug effects. This structure, characterized by a switching function involving the adjoint state, suggests that effective chemotherapy regimens should adapt to both the current tumor state and its predicted evolution.

The mathematical techniques developed in this work extend beyond oncology, with potential applications in numerous biological systems where nonlocal interactions govern emergent behaviors. These include neural tissue dynamics, bacterial biofilm formation, embryogenesis, and even ecosystem population distributions where long-range interactions shape spatial patterns. Within cancer research specifically, several promising extensions warrant investigation. Future work could incorporate immune system responses by coupling our model with immune cell dynamics, allowing exploration of combined chemo-immunotherapy protocols. The framework could also be extended to account for tumor heterogeneity and phenotypic plasticity, including drug resistance mechanisms that evolve during treatment.
\section*{Acknowledgment}
This work was supported by the Centre National pour la Recherche Scientifique et Technique (CNRST) of Morocco and the Laboratoire de Mathématiques Pures et Appliquées (LMPA) at the Université du Littoral Côte d'Opale (ULCO), France.

	\bibliography{ref2}
	\bibliographystyle{abbrv}
	
	\include{annex}
	
\end{document}